\documentclass{amsart}

\usepackage{amsfonts,amssymb,amsthm,amsmath}
\usepackage{graphicx,pst-all}
\usepackage{caption}
\usepackage{amsthm}
\usepackage{graphicx}
\usepackage{tikz}
\usepgflibrary{patterns}
\usepackage{float}
\usepackage{color}
\usepackage{hyperref}
\usetikzlibrary{patterns}

\numberwithin{equation}{section}

\newtheorem{theorem}{Theorem}[section]
\newtheorem{lemma}[theorem]{Lemma}
\newtheorem{definition}[theorem]{Definition}
\newtheorem{corollary}[theorem]{Corollary}
\newtheorem{proposition}[theorem]{Proposition}

\theoremstyle{definition}
\newtheorem{remark}[theorem]{Remark}
\newtheorem{example}[theorem]{Example}

\newcommand{\R}{{\mathbb{R}}}

\DeclareSymbolFont{AMSb}{U}{msb}{m}{n}
\newcommand{\N}{{\mathbb{N}}}
\newcommand{\Z}{{\mathbb{Z}}}
\newcommand{\Q}{Q}
\renewcommand{\P}{{\mathbb{P}}}

\renewcommand{\d}{{\mbox{d}}}

\newcommand{\W}{{\mathbb{W}}}

\newcommand{\EE}{{\mathbb{E}}}

\newcommand{\leb}{{\mathcal{L}}}
\newcommand{\CCost}{{\sf{Cost}}}
\newcommand{\Cost}{{\frak{Cost}}}

\newcommand{\CCo}{{{\sf C}}}

\newcommand{\supp}{{\mbox{supp}}}

\begin{document}
 
\title{Optimal Transport between Random Measures}

\author{Martin Huesmann}
\thanks{M. Huesmann gratefully acknowledges funding through the SFB 611 and a BIGS scholarship.}
\address{Universit\"at Bonn\\
Institut f\"{u}r angewandte Mathematik\\
Endenicher Allee 60\\
53115 Bonn\\
Germany}
\email{huesmann@iam.uni-bonn.de}

\date{}

\begin{abstract}
We study couplings $q^\bullet$ of two equivariant random measures $\lambda^\bullet$ and $\mu^\bullet$ on a Riemannian manifold $(M,d,m)$. Given a cost function we ask for minimizers of the mean transportation cost per volume. In case the minimal/optimal cost is finite and $\lambda^\omega\ll m$ we prove that there is a unique equivariant coupling minimizing the mean transportation cost per volume. Moreover, the optimal coupling is induced by a transportation map, $q^\bullet=(id,T)_*\lambda^\bullet.$ We show that the optimal transportation map can be approximated by solutions to classical optimal transportation problems on bounded regions. In case of $L^p-$cost the optimal transportation cost per volume defines a metric on the space of equivariant random measure with unit intensity.
\end{abstract}
\maketitle

\section{Introduction and Statement of Main Results}
Let  $(M,d,m)$ be a connected smooth non-compact Riemannian manifold with Riemannian distance $d$, and Riemannian volume $m$. Assume that there is a group $G$ of isometries of $M$ acting properly discontinuously, cocompactly and freely on $M$. A random measure $\lambda^\bullet$ on $M$ is a measure valued random variable modeled on some probability space $(\Omega,\mathfrak A,\P).$ We assume that the probability space admits a measurable flow $(\theta_g)_{g\in G}$ which we interpret as the action of $G$  on the support of $\lambda^\omega.$ A random measure $\lambda^\bullet$ is called equivariant if 
$$ \lambda^{\theta_g\omega}(g\ \cdot) = \lambda^\omega(\cdot) \quad \text{for all } \omega\in\Omega, g\in G.$$
We will assume that $\P$ is stationary, that is $\P$ is invariant under the flow $\theta$. In particular, this implies that $\lambda^\bullet(B)\stackrel{d}{=}\lambda^\bullet(gB)$ for any $g\in G$ and Borel set $B$. All random measures will be defined on the \emph{same} probability space.
\medskip\\
We want to extend the theory of optimal transportation to the case of equivariant random measure $\lambda^\bullet, \mu^\bullet$ on $M$. Due to the almost sure infinite mass of $\lambda^\omega$ and $\mu^\omega$ the usual notion of optimality, namely being a minimizer of the total transportation cost, is not meaningful. Therefore, we restrict our investigation to the case of equivariant random measures. For, equivariance allows to transform local quantities into global quantities. To be more precise, if an equivariant coupling can be locally improved it can also be globally improved.
\medskip\\
Hence, given two equivariant random measures $(\lambda^\bullet,\mu^\bullet)$ of equal intensity on $M$, we are interested in couplings $q^\bullet$ of $\lambda^\bullet$ and $\mu^\bullet$, i.e. measure valued random variables $\omega\mapsto q^\omega$ such that for any $\omega\in\Omega$ the measure $q^\omega$ on $M\times M$ is a coupling of $\lambda^\omega$ and $\mu^\omega$. We look for minimizers of the \emph{mean transportation cost}
$$ \mathfrak C(q^\bullet)\ :=\ \sup_{B\in\text{Adm}(M)}\frac1{m(B)}\EE\left[\int_{M\times B} c(x,y)\ q^\bullet(dx,dy)\right],$$
where $\text{Adm}(M)$ is the set of all bounded Borel sets that can be written as the union of ``translates'' of fundamental regions (see section \ref{section optimality}). For example for $M=\R^d, G=\Z^d$ acting by translation, a typical set would be a finite union of unit cubes. We always consider cost functions of the form $c(x,y)=\vartheta(d(x,y))$ for some continuous strictly increasing function $\vartheta:\R_+\to\R_+$ with $\vartheta(0)=0$ and $\lim_{r\to\infty}\vartheta(r)=\infty.$ Additionally, we assume that the classical Monge problem between two compactly supported probability measures $\lambda$ and $\mu$ with $\lambda\ll m$ has a unique solution.
\medskip\\
A coupling $q^\bullet$ of $\lambda^\bullet$ and $\mu^\bullet$ is called \emph{optimal} if it is equivariant and minimizes the mean transportation cost among all equivariant couplings. The set of all equivariant couplings between $\lambda^\bullet$ and $\mu^\bullet$ will be denoted by $\Pi_e(\lambda^\bullet, \mu^\bullet).$ We will show that there always is at least one optimal coupling as soon as the optimal mean transportation cost is finite. A natural question is in which cases  can  we say more about the optimal coupling? When is it unique? Is it possible to construct it? Can we say something about its geometry?  The first main result states

\begin{theorem}\label{main thm 1}
Let $(\lambda^\bullet,\mu^\bullet)$ be two equivariant random measures on M. If the optimal mean  transportation cost is finite
$$\mathfrak c_{e,\infty} \ =\ \inf_{q^\bullet\in\Pi_e(\lambda^\bullet,\mu^\bullet)} \mathfrak C(q^\bullet)\ <\ \infty$$
and  $\lambda^\omega$ is absolutely continuous to the volume measure m for almost all $\omega$, then there is a unique optimal coupling $q^\bullet$ between $\lambda^\bullet$ and $\mu^\bullet$. It can be represented as $q^\omega=(id,T^\omega)_*\lambda^\omega$ for some measurable map $T^\omega:\supp(\lambda^\omega)\to\supp(\mu^\omega)$ measurably only dependent on the $\sigma-$algebra generated by $(\lambda^\bullet,\mu^\bullet)$.
\end{theorem}

In particular, considering  $\lambda^\bullet=m$ being the Riemannian volume measure and $\mu^\bullet$ a point process on $M$ the optimal transportation map $T^\omega$ defines a \emph{fair factor allocation} for $\mu^\bullet$. The inverse map of $T^\omega$ assigns to each point (``center'') $\xi$ of $\mu^\omega$ a set (``cell'') of volume $\mu^\omega(\xi)$. If the point process is simple, all the cells will have mass one. In the  case of $M=\R^d$ and quadratic cost $c(x,y)=|x-y|^2$ all cells will be convex polytopes of volume one, they constitute a Laguerre tessellation (see \cite{lautensack2007}). In the case of linear cost $c(x,y)=|x-y|$ all cells will be starlike with respect to their center, the allocation becomes a Johnson-Mehl diagram (see \cite{142747}). In the light of these results one might interpret the optimal coupling  as a generalized tessellation. If $\mu^\bullet$ is even invariant under the action of $\R^d$ the optimal cost between the Lebesgue measure $\leb$ and $\mu^\bullet$ is given by 
$$\mathfrak c_{e,\infty}\ =\ \EE[\vartheta(|T(0)|)],$$
 recovering a quality factor studied by Peres et alii in the context of allocations (e.g. \cite{phasegravity}).
\medskip\\
Moreover, we prove that the optimal coupling $Q^\infty$, if it is unique, can be obtained as the limit of classical optimal couplings of $\lambda^\bullet$ and $\mu^\bullet$ restricted to bounded sets. For the construction we need to impose an additional growth assumption on the group $G$. The assumptions on the group action imply that $G$ is finitely generated. Let $S$ be a generating set und consider the Cayley graph of $G$ with respect to $S$, $\Delta(G, S)$. Let $\Lambda_r$ denote the closed $2^r$ neighbourhood of the identity of $\Delta(G,S)$. We will assume that $G$ satisfies some strong kind of amenability or otherwise said a certain growth condition, namely
$$\lim_{r\to\infty}\frac{|\Lambda_r\triangle g\Lambda_r|}{|\Lambda_r|}\ =\ 0,$$
for all $g\in G$, where $|\cdot|$ denotes the cardinality and $\triangle$ the symmetric difference. Let $B_0$ be a fundamental region and $B_r=\Lambda_r B_0$. Let $Q_{B_r}$ be the unique optimal \emph{semicoupling} between $\lambda^\bullet$ and $1_{B_r}\mu^\bullet$, that is the unique optimal coupling between $\rho\cdot\lambda^\bullet$ and $1_{B_r}\mu^\bullet$  for some optimal choice of density $\rho$. Put
$$ \tilde Q^r_g \ :=\  \frac1{|\Lambda_r|}\sum_{h\in g\Lambda_r} Q_{hB_r}. $$

\begin{figure*}
\begin{center}
 \includegraphics[scale=1.2]{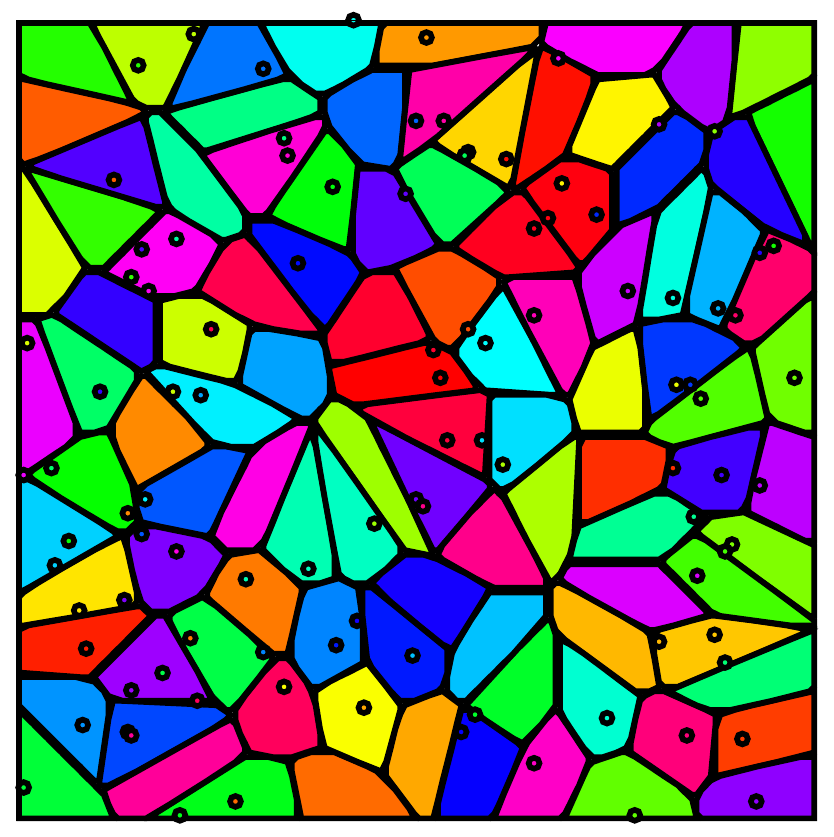}
\end{center}
\caption{Coupling of Lebesgue and 100 points in the cube with $c(x,y)=|x-y|^2.$}
\end{figure*}

\begin{figure*}
 \begin{center}
  \includegraphics[scale=1.2]{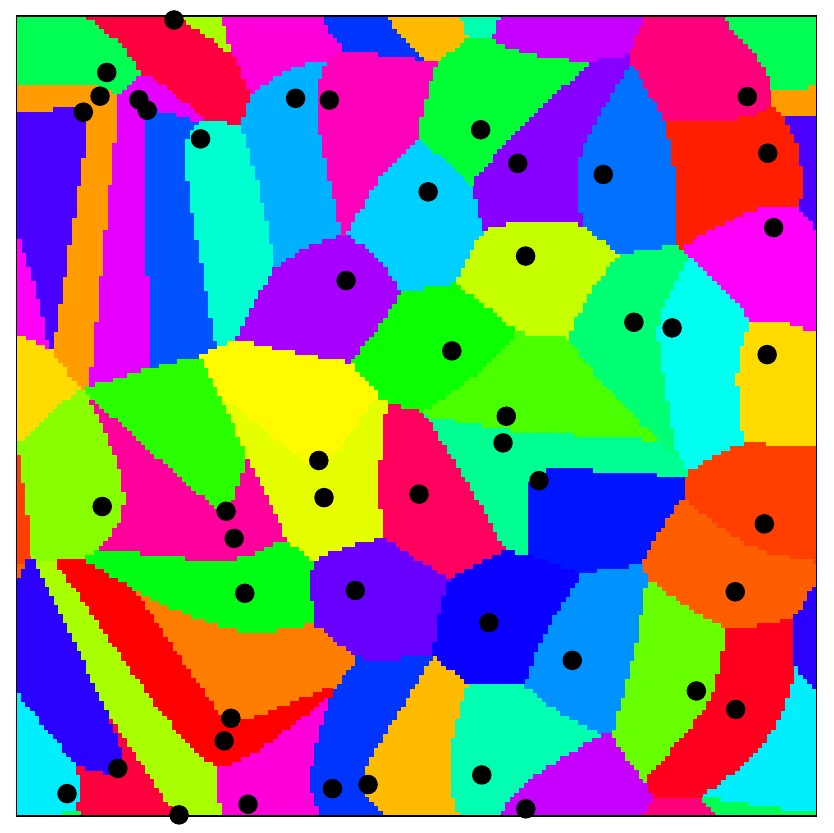}
 \end{center}
\caption{Coupling of volume measure and 49 points on a torus with cost function $c(x,y)=d(x,y).$}
\end{figure*}

\begin{theorem}\label{main thm 2}
Let $(\lambda^\bullet,\mu^\bullet)$ be two equivariant random measures on M, such that the optimal mean transportation cost are finite, $\mathfrak c_{e,\infty}<\infty.$ Assume, that $\lambda^\omega$ is absolutely continuous to the volume measure m for almost all $\omega$. Then, for every $g\in G$ 
$$ \tilde Q^r_g\ \to\ Q^\infty \quad \text{vaguely}$$
in $\mathcal M(M\times M\times\Omega).$
\end{theorem}

For the proof of this theorem the assumption of absolute continuity is only needed to ensure uniqueness of $Q_{gB_r}$ and $Q^\infty$. If we do not have absolute continuity but uniqueness of $Q_{gB_r}$ and $Q^\infty$ the same theorem with the same proof holds. 
\medskip\\
In the case of absolute continuity we can even say a bit more and get rid of the mixing. The unique optimal coupling is given by a map, that is
$$Q^\infty=(id,T)_*\lambda^\bullet.$$
Moreover, the optimal semicoupling $Q_{gB_r}$ is  given by
$$Q_{gB_r}=(id,T_{g,r})_*(\rho_{g,r}\lambda^\bullet),$$
for some measurable map $T_{g,r}$ and some density $\rho_{g,r}$. Then, we have
\begin{theorem}\label{main thm 3}
For every $g\in G$ 
$$T_{g,r}\ \to T\quad \text{ locally in } \lambda^\bullet\otimes \P \text{ measure }.$$  
\end{theorem}

Analogous results will be obtained in the more general case of optimal semicouplings between $\lambda^\bullet$ and $\mu^\bullet$ where $\lambda^\bullet$ has intensity one and $\mu^\bullet$ has intensity $\beta\in (0,\infty)$ (see Theorem \ref{uniqueness}, Theorem \ref{Q infty}, Proposition \ref{t-map convergence} and section \ref{s:other}). In the case $\beta \leq 1$, $\lambda^\bullet$ is allowed to not transport all of its mass. There will be some areas from which nothing is transported and the $\mu^\bullet$ mass can choose its favorite $\lambda^\bullet$ mass. In the case $\beta \geq 1$ the situation is the opposite. There is too much $\mu^\bullet$ mass. Hence, $\lambda^\bullet$ can choose its favorite $\mu^\bullet$ mass and some part of the $\mu^\bullet$ mass will not be satisfied, that is they will not get enough or even any of the $\lambda^\bullet$ mass.
\medskip\\
As a special case of our result, we recover the results by Huesmann and Sturm in \cite{huesmann2010optimal}. They studied couplings between the Lebesgue measure and an equivariant point process of intensity $\beta\in (0,1]$. They showed that there is a unique optimal semicoupling and also proved an approximation result by solutions to transport problems on bounded regions.
\medskip\\
Furthermore, in \cite{huesmann2010optimal} necessary and sufficient conditions have been derived implying the finiteness of the mean asymptotic transportation cost in the case of transporting the Lebesgue measure to a Poisson point process. By applying the same techniques similar estimates can be achieved for the case of a compound Poisson process with iid weights $(X_i)_{i\in \N},$ i.e. $\mu^\omega=\sum_{i\in\N} X_i\delta_{z_i}.$ In the case that $X_1$ is exponentially distributed it is possible to use the algorithm by Mark\'o and Timar \cite{marko2011poisson} to construct an equivariant coupling between the Lebesgue measure and $\mu^\bullet$ with optimal tail, i.e. with finite mean transportation cost for the cost function $c(x,y)=\exp(\kappa\cdot |x-y|^d)$ for some positive $\kappa$ and dimension $d\geq 3.$ 
\medskip\\
In the case of $c(x,y)=d^p(x,y)$ for $p\in[1,\infty)$ we write the optimal mean transportation cost between $\lambda^\bullet$ and $\mu^\bullet$ as $\W_p^p(\lambda^\bullet,\mu^\bullet)$, i.e.
$$\W_p^p(\lambda^\bullet,\mu^\bullet)\ =\ \inf_{q^\bullet\in \Pi_e(\lambda^\bullet,\mu^\bullet)} \mathfrak C(q^\bullet).$$
Let 
$$\mathcal P_p=\{\text{equivariant random measures $\mu^\bullet$ on $M$ with unit intensity s.t. $\W_p(m,\mu^\bullet)<\infty$}\}$$
Then $\W_p$ defines a metric on $\mathcal P_p$ which implies the vague convergence of the Campbell measures (see Propositions \ref{metric} and \ref{half stability}). Moreover, if we take two sequences of equivariant random measures $(\lambda^\bullet_n)_{n\in\N}$ and $(\mu^\bullet_n)_{n\in\N}$ such that their Campbell measures converge vaguely to some equivariant random measures $\lambda^\bullet, \mu^\bullet$, i.e.
$$ \lambda^\bullet_n\P\to \lambda^\bullet\P, \qquad \mu_n^\bullet\P\to\mu^\bullet\P,$$
and the optimal mean transportation cost converge $\W_p(\lambda^\bullet_n, \mu^\bullet_n)\to\W_p(\lambda^\bullet, \mu^\bullet)$, then the optimal semicouplings between $\lambda_n^\bullet$ and $\mu^\bullet_n$ converge to the optimal semicoupling between $\lambda^\bullet$ and $\mu^\bullet$ (see Proposition \ref{nearly stability}).
\medskip\\
It is clear that if the restriction of $\lambda^\omega\ll m$ is lifted there will not be a unique semicoupling in general. However, in the case of two independent Poisson processes on $\R^d$ we conjecture that there is a unique optimal semicoupling given that the mean transportation cost is finite. In particular, this would imply that the optimal coupling between two Poisson processes is a matching.
\medskip\\
Matchings of two independent Poisson processes have been intensely studied in \cite{holroyd2009geometric, matching09}. However, there are still a couple of open questions. Solving the conjecture on optimal couplings between two Poisson processes might help solve some of them. In \cite{last2008invariant} Last and Thorisson studied equivariant transports between random measures in a rather general setting. In the recent article \cite{last2011unbiased} Last, M\"orters and Thorisson constructed an equivariant transport between two diffuse random measures to study unbiased shifts of Brownian motion. They also derive some moment estimates on the typical transport distance. Fair allocations have been studied and constructed, e.g. in \cite{extra-heads} \cite{stable-marriage, gravity, marko2011poisson} and references therein. However, \cite{huesmann2010optimal} is, to our knowledge, so far the only article studying couplings of two random measures under the additional requirement of being cost minimizing, that is optimal.

\subsection{Outline}
In section \ref{s:su} we introduce the setting and objects we work with. Section \ref{s:bd semi} is devoted to the proof of a key technical lemma, the existence and uniqueness of optimal semicouplings on bounded sets. In section \ref{s:u} we prove Theorem \ref{main thm 1}. Theorem  \ref{main thm 2} and Theorem \ref{main thm 3} are proved in section  \ref{s:construction}. In all these sections we always assume that the second marginal has intensity $\beta\leq 1$. In section \ref{s:other} we treat the case of $\beta\geq 1$. In section \ref{s: cost estimate} we state the estimates on the compound Poisson process. Finally in section \ref{s: stability} we show that $\W_p$ defines a metric on $\mathcal P_p$ and prove the stability result.

\section{Set-up}\label{s:su}
In this section we will explain the general set-up, some basic concepts and derive the first result, a general existence result by a compactness argument.

\subsection{The setting}
From now on we will always assume to work in the following setting. $(M, d, m)$ will denote a complete connected smooth non-compact Riemannian manifold with Riemannian distance $d$ and Riemannian volume measure $m$. The Borel sets on $M$ will be denoted by $\mathcal B(M)$. Given a map $S$ and a measure $\rho$ we denote the push forward of $\rho$ under S by $S_*\rho$, i.e. $S_*\rho(A)=\rho(S^{-1}(A))$ for any Borel set $A$. Given any product $X=\Pi_{i=1}^n X_i$ of measurable spaces, the projection onto the i--th space will be denoted by $\pi_i$. Given a set $A\subset M$ its complement will be denoted by $\complement A$ and the indicator function of $A$ by $1_A$.\\

We will assume that there is a   group $G$ of isometries acting on $M$. For a set $A\subset M$ we write $\tau_gA\ :=\ gA\ =\ \{ga:\ a\in A\}$.  For a point $x\in M$ its \emph{orbit} under the group action of G is defined as $Gx=\{gx\ :\ g\in G\}.$ Its \emph{stabilizer} is defined as $G_x=\{g\in G\ :\ gx=x\}$ the elements of G that fix x.

\begin{definition}[Group action]
 Let G act on M. We say that the action is
\begin{itemize}
 \item \emph{properly discontinuous} if for any $x\in M$ and any compact $K\subset M$ $gx\in K$ for only finitely many $g\in G$.
\item \emph{cocompact} if $M/G$ is compact in the quotient topology.
\item \emph{free} if $gx=x$ for one $x\in M$ implies $g=id$, that is the stabilizer for every point is trivial.
\end{itemize}
\end{definition}

We will assume that the group action is properly discontinuous, cocompact and free. By Theorem 3.5 in \cite{bowditch2006course} this already implies that G is finitely generated and therefore countable.

\begin{definition}[Fundamental region]
A measurable subset $B_0\subset M$ is defined to be a fundamental region for G if
\begin{itemize}
 \item[i)] $\bigcup_{g\in G} gB_0 = M$
\item[ii)] $B_0\cap gB_0=\emptyset$ for all $id\neq g\in G$.
\end{itemize}
The family $\{gB_0\ :\ g\in G\}$ is also called tessellation of M.
\end{definition}

There are many different choices of fundamental regions. We will choose a special one, namely a certain subset of the Dirichlet region with respect to some fixed point p. However, each fundamental region has the same volume  and therefore defines a tiling of M in pieces of equal volume. Indeed, we have the following Lemma.

\begin{lemma}\label{fundamental region have same mass}
 Let $F_1$ and $F_2$ be two fundamental regions for G. Assume $m(F_1)<\infty$. Then $m(F_1)=m(F_2).$
\end{lemma}
\begin{proof}
 As $F_1\cap gF_2$ and $F_1\cap hF_2$ are disjoint for $g\neq h$ by the defining property of fundamental regions we have
$$ m(F_1) = \sum_{g\in G} m(F_1\cap gF_2) = \sum_{g\in G} m(g^{-1}F_1\cap F_2) = m(F_2).$$
\end{proof}

By scaling of the volume measure $m$ we can assume that $m(B_0)=1$. This assumption is just made to simplify some notations.

As G is finitely generated, there are finitely many elements $a_1,\ldots,a_k \in G$ such that every $g\in G$ can be written as a word in these letters and their inverses. The set $S=\{a_1,\ldots, a_k\}$ is called a generating set. The generating set is not unique, e.g. $\Z$ is generated by $\{1\}$ but also by $\{2,3\}$. We will fix one finite generating set for G. It does not matter which one as the results will be independent from the specific choice. \\
Given the generating set S. We can construct a graph $\Delta\ =\ \Delta(G, S)$ as follows. Put $V(\Delta)=G$ as the vertices. For each $g\in G$ and $a\in S$ we connect $g$ and $ag$ by a directed edge labeled with a. The same edge with opposite orientation is labeled by $a^{-1}$. This gives a regular graph of degree $2|S|$. We endow $\Delta$ with the word metric $d_\Delta$ which coincides with the usual graph distance.

\begin{definition}[Cayley graph]
 If S is a generating set of G, then $\Delta(G, S)$ is called Cayley graph of G with respect to S.
\end{definition}

We denote the closed $2^r$ neighbourhood of the identity element in $\Delta$ by $\Lambda_r$, that is $\Lambda_r\ =\ \{g\in G:\ d_\Delta(1,g)\leq 2^r\}$. The boundary of $\Lambda_r$ is defined as $\partial \Lambda_r=\{h\notin \Lambda_r\ : \exists g\in\Lambda_r \text{ s.t. }\ d_{\Delta}(h,g)=1\}.$ By $B_r$ we denote the range of the action of $\Lambda_r$ on the fundamental domain $B_0$, that is $B_r\ =\ \bigcup_{g\in \Lambda_r} gB_0$.\\

We will need to control the mass that is close to the boundary of $B_r$, that is the growth of $B_r$. In section \ref{s:construction}, we will assume that 
$$\lim_{r\to\infty}\frac{|\Lambda_r\triangle g\Lambda_r|}{|\Lambda_r|}\ =\ 0\quad \forall\ g\in G.$$

\medskip

Several times we will use a rather simple but very powerful tool, the \emph{mass transport principle}. It already appeared in the proof of Lemma \ref{fundamental region have same mass}. It is a kind of conservation of mass formula for invariant transports.
\begin{lemma}[mass transport principle]
  Let $f:G\times G\to \R_+$ be a function which is invariant under the diagonal action of G, that is $f(u,v)=f(gu,gv)$ for all $g,u,v\in G$. Then we have 
$$\sum_{v\in G} f(u,v)\ =\  \sum_{v\in G} f(v,u).$$
\end{lemma}
\begin{proof}
 $$\sum_{v\in G} f(u,v)\ =\ \sum_{g\in G} f(u,gu)\ =\ \sum_{g\in G} f(g^{-1} u,u)\ =\ \sum_{v\in G} f(v,u).$$
\end{proof}

For a more general version we refer to \cite{benjamini1999group} and \cite{last2008invariant}. 

\medskip

Recall the disintegration theorem for finite measures (e.g. see Theorem 5.1.3 in \cite{AGS} or III-70 in \cite{dellacherie1978probabilities}).
\begin{theorem}[Disintegration of measures]\label{disintegration theorem}
 Let X, Y be Polish spaces, and let $\gamma$ be a finite Borel measure on $X\times Y$. Denote by $\mu$ and $\nu$ the marginals of $\gamma$ on the first and second factor respectively. Then, there exist two measurable families of probability measures $(\gamma_x)_{x\in X}$ and $(\gamma_y)_{y\in Y}$ such that
$$\gamma(dx,dy)=\gamma_x(dy)\mu(dx)=\gamma_y(dx)\nu(dy).$$
\end{theorem}

\subsection{Couplings and Semicouplings}\label{section 2.2}

For each Polish space $X$ (i.e. complete separable metric space) the set of Radon measures on $X$ -- equipped with its Borel $\sigma$-field -- will be denoted by $\mathcal M(X)$. 
Given any ordered pair of Polish spaces $X,Y$ and  measures $\lambda\in \mathcal M(X), \mu\in\mathcal M(Y)$  we say that a measure $q\in\mathcal M(X\times Y)$ is a \emph{semicoupling} of $\lambda$ and $\mu$, briefly $q\in \Pi_{s}(\lambda,\mu)$,
iff
the (first and second, resp.) marginals satisfy $$(\pi_1)_\ast q\leq \lambda,\qquad(\pi_2)_{\ast}q=\mu,$$ that is, iff
$q(A\times Y)\le \lambda(A)$ and $q(X\times B)=\mu(B)$ for all Borel sets $A\subset X, B\subset Y$.
The semicoupling $q$ is called \emph{coupling}, briefly $q\in \Pi(\lambda,\mu)$, iff in addition
$$(\pi_1)_\ast q= \lambda.$$

\medskip

See also \cite{Figalli2010optimal} for the related concept of \emph{partial coupling}.

\subsection{Random measures on M}\label{random measures on M}
We endow $\mathcal M(M)$   with the vague topology.  The next Lemma summarizes some basic facts about vague topology (e.g. see \cite{Kallenberg1997foundations} or \cite{bauer2001measure})
\begin{lemma}[vague topology]\label{vague topology}
 Let X be a locally compact second countable Haussdorff space. Then,
\begin{itemize}
 \item[i)] $\mathcal M(X)$ is a Polish space in the vague topology.
\item[ii)]  $A\subset \mathcal M(X)$ is vaguely relatively compact iff $\sup_{\mu\in A} \mu(f) <\infty$ for all $f\in C_c(X).$
\item[iii)] If $\mu_n\stackrel{v}{\to}\mu$ and $B\subset X$ relatively compact with $\mu(\partial B)=0$ then $\mu_n(B)\to\mu(B).$
\end{itemize}

\end{lemma}

 The action of G on M induces an action of G on $\mathcal M(M\times\ldots\times M)$ by push forward with the map $\tau_g$: 
$$(\tau_g)_*\lambda(A_1,\ldots,A_k)\ =\ \lambda((g^{-1}(A_1),\ldots,g^{-1}(A_k))\quad  \forall A_1,\ldots A_k\in\mathcal B(M), k\in \N .$$

A random measure on M is a random variable $\lambda^\bullet$ (the notation with the ``$\bullet$'' is intended to make it easier to distinguish random and non-random measures) modeled on some probability space $(\Omega, \mathfrak{A}, \mathbb P)$ taking values in $\mathcal M(M)$. It can also be regarded as a kernel from $\Omega$ to M. Therefore, we write either $\lambda^\omega(A)$ or $\lambda(\omega,A)$ depending on which property we want to stress. For convenience, we will assume that $\Omega$ is a compact metric space and $\frak{A}$ its completed Borel field. These technical assumptions are only made to simplify the presentation. \\

A point process is a random measure $\mu^\bullet$ taking values in the (vaguely closed) subset of all locally finite counting measures on M. It is called simple iff $\mu^\omega(\{x\})\in\{0,1\}$ for every $x\in M$ and a.e. $\omega\in\Omega$. We call a random measure $\lambda^\bullet$ absolutely continuous iff it is absolutely continuous to the volume measure m on M for a.e. $\omega\in\Omega.$ It is called diffusive iff it has no atoms almost surely. The \emph{intensity measure} of a random measure $\lambda^\bullet$ is a measure on $M$  defined by $A\mapsto \EE[\lambda^\bullet(A)].$ \\
The class of all relatively compact sets in $\mathcal B(M)$ will be denoted by $\hat{\mathcal B}$. For a random measure $\lambda^\bullet$ its class of stochastic continuity sets is defined by $\hat{\mathcal B}_{\lambda^\bullet}\ =\ \{A\in \hat{\mathcal B} : \lambda^\bullet(\partial A)=0\ a.s.\}.$ Convergence in distribution and tightness in $\mathcal M(M)$ can be characterized by
\begin{lemma}[tightness of random measures]\label{general tightness}
 Let $\lambda_1^\bullet,\lambda_2^\bullet,\ldots$ be random measures on M. Then the sequence $(\lambda_n^\bullet)_{n\in\N}$ is relatively compact in distribution iff $(\lambda_n^\bullet(A))_{n\in\N}$ is tight in $\R_+$ for every $A\in\hat{\mathcal B}$.
\end{lemma}
\begin{theorem}[convergence of random measures]\label{convergence of rm}
Let $\lambda^\bullet, \lambda_1^\bullet, \lambda_2^\bullet,\ldots$ be random measures on M. Then, these conditions are equivalent:
\begin{itemize}
 \item[i)] $\lambda_n^\bullet\stackrel{d}\to \lambda^\bullet$
\item[ii)]  $\lambda_n^\bullet (f)\stackrel{d}\to \lambda^\bullet (f)$ for all $f\in C_c(M)$
\item[iii)] $(\lambda_n^\bullet(A_1),\ldots,\lambda_n^\bullet(A_k))\stackrel{d}\to (\lambda^\bullet(A_1),\ldots,\lambda^\bullet(A_k))$ for all $A_1,\ldots,A_k\in\hat{\mathcal B}_{\lambda^\bullet}, k\in\N.$
\end{itemize}
If $\lambda^\bullet$ is a simple point process or a diffusive random measure, it is also equivalent that 
\begin{itemize}
 \item[iv)] $\lambda_n^\bullet(A)\stackrel{d}\to\lambda^\bullet(A)$ for all $A\in\hat{\mathcal B}_{\lambda^\bullet}.$
\end{itemize}
\end{theorem}
For the proof of these statements we refer to Lemma 14.15 and Theorem 14.16 of \cite{Kallenberg1997foundations}.

Just as in Lemma 11.1.II of \cite{daley2007introduction} we can derive the following result on continuity sets of a random measure $\lambda^\bullet$:
\begin{lemma}\label{cont set}
 Let $\lambda^\bullet$ be a random measure on M, $A \in \mathcal B(M)$ be bounded and $(A)_r$ be the r-neighbourhood of A in M. Then for all but a countable set of $r\in\R_+$ we have $(A)_r\in  \hat{\mathcal B}_{\lambda^\bullet}$.
\end{lemma}

\medskip

A random measure $\lambda^\bullet:\Omega\to\mathcal M(M)$ is called G-invariant or just invariant if the distribution of $\lambda^\bullet$ is invariant under the action of G, that is, iff
 $$(\tau_g)_*\lambda^\bullet \quad \stackrel{(d)}=\quad \lambda^\bullet$$
for all $g\in G$. A random measure $q^\bullet:\Omega\to\mathcal M(M\times M)$ is called invariant if its distribution is invariant under the diagonal action of G.\\
If $(\Omega, \mathfrak A)$ admits a measurable flow $\theta_g:\Omega\to\Omega, g\in G, $ that is a measurable mapping $(\omega, g)\mapsto \theta_g\omega$ with $\theta_0$ the identity on $\Omega$ and 
$$\theta_g \circ \theta_h \ =\ \theta_{gh},\quad g,h\in G,$$
then a random measure $\lambda^\bullet:\Omega\to\mathcal M(M)$ is called G-equivariant or just equivariant iff
$$ \lambda(\theta_g \omega,gA)\ =\ \lambda(\omega,A),$$
for all $g\in G, \omega\in\Omega, A\in\mathcal B(M)$. We can think of $\lambda(\theta_g\omega,\cdot)$ as $\lambda(\omega,\cdot)$ shifted by g. Indeed, let $\mathfrak M$ be the cylindrical $\sigma-$algebra generated by the evaluation functionals $A\mapsto \mu(A), A\in \mathcal B(M), \mu\in\mathcal M$. As in example 2.1 of \cite{last2008invariant}, consider the measurable space $(\mathcal M, \mathfrak M)$ and define for $\mu\in\mathcal M, g\in G$ the measure $\theta_g\mu(A)=\mu(g^{-1}A)$. Then, $\{\theta_g, g\in G\}$ is a measurable flow and the identity is an equivariant measure. A random measure $q^\bullet :\Omega\to\mathcal M(M\times M)$ is called equivariant iff
$$ q^{\theta_g\omega}(gA,gB)\ =\  q^{\omega}(A,B),$$
for all $g\in G, \omega\in\Omega, A, B \in\mathcal B(M).$ 

\begin{example}\label{equivariance of maps}
 Let $q^\bullet$ be an equivariant random measure on $M\times M$ given by $q^\omega=(id,T^\omega)_*\lambda^\omega$ for some measurable map $T^\bullet$ and some equivariant random measure $\lambda^\bullet.$ The equivariance condition 
$$\int_A 1_B(y) \delta_{T^{\theta_g\omega}(gx)}(d(gy))\lambda^{\theta_g\omega}(dx)=q^{\theta_g\omega}(gA,gB)=q^\omega(A,B)=\int_A 1_B(y)\delta_{T^\omega(x)}(dy)\lambda^\omega(dx),$$
translates into an equivariance condition for the transport maps:
$$ T^{\theta_g\omega}(gx)=gT^\omega(x).$$
\end{example}

A probability measure $\P$ is called \emph{stationary} iff 
$$\P\circ\theta_g=\P$$
 for all $g\in G$. Given a measure space $(\Omega,\mathfrak A)$ with a measurable flow $(\theta_g)_{g\in G}$ and a stationary probability measure $\P$ any equivariant measure is automatically invariant. The advantage of this definition is that the sum of equivariant measures is again equivariant, and therefore also invariant. The sum of two invariant random measures does not have to be invariant (see Remark \ref{invariant plus invariant}).\\

We say that a random measure $\lambda^\bullet$ has subunit intensity iff $\EE[\lambda^\bullet(A)]\leq m(A)$ for all $A\in\mathcal B(M)$. If equality holds in the last  statement we say that the random measure has unit intensity. An invariant random measure has subunit (or unit) intensity iff its intensity
$$\beta\ =\ \EE[\lambda^\bullet(B_0)]$$
is $\leq 1$ (or $=1$ resp.). Given a random measure, the measure $(\lambda^\bullet\mathbb P)(dy,d\omega):=\lambda^\omega(dy)\,\mathbb P(d\omega)$ on $M\times \Omega$ is called  \emph{Campbell measure} of the random measure $\lambda^\bullet$.\\

\begin{example}
\begin{itemize}
 \item[i)] The Poisson point process with intensity measure m. It is characterized by
\begin{itemize}
\item for each Borel set $A\subset M$ of finite volume the random variable $\omega\mapsto \mu^{\omega}(A)$ is Poisson distributed with parameter $ m(A)$ and
\item for disjoint sets $A_1,\ldots A_k\subset M$ the random variables $\mu^{\omega}(A_1),\ldots,\mu^{\omega}(A_k)$ are independent.
\end{itemize}
It can be written as 
$$\mu^\omega=\sum_{\xi\in \Xi(\omega)}\delta_\xi$$
    with some countable set $\Xi(\omega)\subset M$ without accumulation points.
\item[ii)] The compound Poisson process is a Poisson process with random weights instead of unit weights. It is compounded with another distribution giving the weights of the different atoms. It can be written as
$$\mu^\omega=\sum_{\xi\in \Xi(\omega)} X_\xi \delta_\xi$$
for some iid sequence $(X_\xi)_{\xi\in\Xi(\omega)}$ independent of the Poisson point process. For example one could take $X_\xi$ to be a Poisson random variable or an exponentially distributed random variable. If $X_\xi$ has distribution $\gamma$ we say $\mu^\bullet$ is a $\gamma-$compound Poisson process.
\end{itemize}

\end{example}

From now on we will always assume that we are given two equivariant random measures $\lambda^\bullet$ and $\mu^\bullet$ modeled on some probability space $(\Omega,\mathfrak A,\P)$ admitting a measurable flow $(\theta_g)_{g\in G}$ such that $\P$ is stationary. We will assume that $\Omega$ is a compact metric space. Moreover, we will assume that $\lambda^\bullet$ is absolutely continuous and $\lambda^\bullet$ and $\mu^\bullet$ are almost surely not the zero measure. Note that the invariance implies that $\mu^\omega(M)=\lambda^\omega(M)=\infty$ for almost every $\omega$ (e.g. see Proposition 12.1.VI in \cite{daley2007introduction}).

\subsection{Semicouplings of $\lambda^\bullet$ and $\mu^\bullet$}
A semicoupling of the random measures $\lambda^\bullet$ and $\mu^\bullet$ is a measurable map $q^\bullet: \ \Omega\to\mathcal M(M\times M)$ s.t. for $\mathbb P$-a.e. $\omega\in\Omega$
$$q^\omega\ \mbox{ is a semicoupling of } \lambda^\omega \mbox{ and } \mu^\omega.$$
Its Campbell measure is given by $Q=q^\bullet\P\in\mathcal M(M\times M\times \Omega)$. $Q$ is a semicoupling between the Campbell measures $\lambda^\bullet\P$ and $\mu^\bullet\P$ in the sense that
$$ Q(M\times \cdot \times \cdot )\ =\ \mu^\bullet\P\ \text{ and } \ Q(\cdot\times M\times\cdot)\ \leq\ \lambda^\bullet\P.$$
$Q$ could also be regarded as semicoupling between $\lambda^\bullet\P$ and $\mu^\bullet\P$ on $M\times\Omega\times M\times\Omega$ which is concentrated on the diagonal of $\Omega\times \Omega$. It could be interesting to relax this last condition on $Q$ and allow different couplings of the randomness. However, we will not do so and only consider semicouplings of $\lambda^\bullet\P$ and $\mu^\bullet\P$ that are concentrated on the diagonal of $\Omega\times\Omega$. We will always identify these semicouplings with measures on $M\times M\times \Omega$. 

Given such a semicoupling $Q\in\mathcal M(M\times M\times \Omega)$ we can disintegrate (see Theorem \ref{disintegration theorem}) $Q$ to get a measurable map $q^\bullet: \ \Omega\to\mathcal M(M\times M)$ which is a semicoupling of $\lambda^\bullet$ and $\mu^\bullet$.

\medskip

According to this one-to-one correspondence between $q^\bullet$ --- semicoupling of $\lambda^\bullet$ and $\mu^\bullet$ --- and $Q=q^\bullet\mathbb P$ --- semicoupling  of $\lambda^\bullet \mathbb P$ and $\mu^\bullet\mathbb P$ ---  we will freely switch between them.  And quite often, we will simply speak of \emph{semicouplings of $\lambda^\bullet$ and $\mu^\bullet$}.

\medskip

We denote the set of all semicouplings between $\lambda^\bullet$ and $\mu^\bullet$ by $\Pi_s(\lambda^\bullet,\mu^\bullet)$. The set of all equivariant semicouplings between $\lambda^\bullet$ and $\mu^\bullet$ will be denoted by $\Pi_{es}(\lambda^\bullet,\mu^\bullet)$.

\medskip

A \emph{factor} of some random variable X is a random variable Y which is measurable with respect to $\sigma(X).$ This is equivalent to the existence of a deterministic function $f$ with $Y=f(X)$. In other words, a factor is a rule such that given X we can construct Y. A \emph{factor semicoupling} is a semicoupling of $\lambda^\bullet$ and $\mu^\bullet$ which is a factor of $\lambda^\bullet$ and $\mu^\bullet$.

\subsection{The Monge-Kantorovich problem}
Let  $\lambda, \mu$ be two probability measures on $M$. Moreover, let a cost function $c: M\times M\to\R$ be given. The Monge-Kantorovich problem is to find a minimizer of
$$\int_{M\times M} c(x,y)\ q(dx,dy)$$
among all couplings $q$ of $\lambda$ and $\mu$. A minimizing coupling is called \emph{optimal coupling}. If the optimal coupling $q$ is induced by a transportation map, i.e. $q=(id,T)_*\lambda$, we say that $q$ is a solution to the Monge problem. There are rather general existence and uniqueness results for optimal couplings. We always assume that the cost function $c(x,y)=\vartheta(d(x,y))$ is such that there is a unique solution to the Monge problem between $\lambda$ and $\mu$ whenever $\lambda\ll m.$ For conditions on $\vartheta$ such that this assumption is satisfied we refer to section \ref{s:bd semi}.
\medskip\\
It can be shown that any optimal coupling is concentrated on a $c-$cyclical monotone set. A set $A\subset X\times X$ is called $c-$cyclical monotone if for all $n\in\N$ and $(x_i,y_i)_{i=1}^n\in A^n$ it holds that
$$ \sum_{i=1}^n c(x_i,y_i) \leq \sum_{i=1}^n c(x_i,y_{i+1}),$$
where $y_1=y_{n+1}.$ If the cost function is reasonably well behaved (continuous is more than sufficient, see \cite{betterplans}), also the reverse direction holds. Any coupling which is concentrated on a $c-$cyclical monotone plan is optimal. For further details and applications of mass transport theory we refer to \cite{Rachev-Ruesch, Villani1, villani2009optimal}

\subsection{Cost functionals}

Throughout this article,
$\vartheta$ will be a strictly increasing, continuous function from $\mathbb R_+$ to $\mathbb R_+$ with $\vartheta(0)=0$ and $\lim\limits_{r\to\infty}\vartheta(r)=\infty$.
Given a \emph{scale function} $\vartheta$ as above we define the \emph{cost function}
$$c(x,y)=\vartheta\left(d(x,y)\right)$$
on $M \times M$, the \emph{cost functional} $$\CCost(q)=\int_{M\times M}c(x,y)\, q(dx,dy)$$
on $\mathcal M(M\times M)$
and the \emph{mean cost functional}
$$\Cost(Q)=\int_{M\times M\times\Omega}c(x,y)\ Q(dx,dy,d\omega)$$
on $\mathcal M(M\times M\times\Omega)$.

We have the following basic result on existence and uniqueness of optimal semicouplings the proof of which is deferred to section \ref{s:bd semi}. The first part of the theorem, the existence and uniqueness of an optimal semicoupling, is very much in the spirit of an analogous result by Figalli \cite{Figalli2010optimal} on existence and (if enough mass is transported) uniqueness of an optimal partial coupling. However, in our case the second marginal is arbitrary whereas in \cite{Figalli2010optimal} it is absolutely continuous.

\begin{theorem}\label{eu:Q+q}
(i) For each bounded Borel set $A\subset M$ there exists a unique semicoupling $\Q_A$ of $\lambda^\bullet\P$ and $(1_A\mu^\bullet)\mathbb P$ which minimizes the mean cost functional $\Cost(.)$.

(ii) The measure  $\Q_A$ can be disintegrated as $\Q_A(dx,dy,d\omega):=q_A^\omega(dx,dy)\,{\mathbb P}(d\omega)$ where for $\mathbb P$-a.e. $\omega$ the measure $q_A^\omega$ is the unique minimizer of the cost functional $\CCost(.)$ among the semicouplings of $\lambda^\omega$ and $1_A\mu^\omega$.

(iii) $\Cost(\Q_A)=\int_\Omega\CCost(q_A^\omega)\, \mathbb P(d\omega).$
\end{theorem}

For  a bounded Borel set $A\subset M$ , the \emph{transportation cost on $A$} is given by the random variable $\CCo_{A}:\Omega\to[0,\infty]$ as $$\CCo_{A}(\omega):=\CCost(q_{A}^\omega)=\inf\{\CCost(q^\omega):\ q^\omega\text{ semicoupling of $\lambda^\omega$ and $1_A\,\mu^\omega$}\}.$$

\begin{lemma}\label{super}
\begin{enumerate}
	\item If $A_1,\ldots,A_n$ are disjoint then $\forall \omega\in\Omega$ $$\CCo_{\bigcup\limits_{i=1}^nA_i}(\omega)\quad\geq \quad \sum_{i=1}^n \CCo_{A_i}(\omega)$$
	\item If  $A_1=gA_2$ for some $g\in G$, then $\CCo_{A_1}$ and $\CCo_{A_2}$ are identically distributed.
\end{enumerate}
\end{lemma}

\begin{proof}
Property (ii)  follows directly from the   joint invariance of $\lambda^\bullet$ and $\mu^\bullet$.  The intuitive argument for (i) is, that minimizing the cost on $\bigcup_i A_i$ is more restrictive than doing it separately on each of the $A_i$. The more detailed argument is the following.
Given any semicoupling $q^\omega$ of $\lambda^\omega$ and $1_{\bigcup_iA_i} \mu^\omega$ then for each $i$ the measure
$q_i^\omega:=1_{M\times A_i}q^\omega$ is a semicoupling of $\lambda^\omega$ and $1_{A_i} \mu^\omega$. Choosing $q^\omega$ as the minimizer of $\CCo_{\bigcup\limits_{i=1}^nA_i}(\omega)$ yields
$$\CCo_{\bigcup_{i}A_i}(\omega)=\CCost(q^\omega)=\sum_i\CCost(q_i^\omega)\ge\sum_{i} \CCo_{A_i}(\omega).$$
\end{proof}

\subsection{Standard tessellations}

In this section, we construct the fundamental region $B_0$ and thereby a tessellation or a tiling of $M$. We will call this tessellation a standard tessellation. The specific choice of fundamental domain is not really important for us. However, we will choose one to fix ideas.

We now define the Dirichlet region. To this end let $p\in M$ be arbitrary. Due to the assumption of freeness, the stabilizer of p is trivial. Construct the Voronoi tessellation with respect to Gp, the orbit of p. The cell containing p is the Dirichlet region. 

\begin{definition}[Dirichlet region]
 Let $p\in M$ be arbitrary. The Dirichlet region of G centered at p is defined by
$$ D_p(G)\ =\ \{x\in M\ :\ d(x,p)\leq d(x,gp)\ \forall g\in G\}.$$
\end{definition}

From now on we will fix p and write for simplicity of notation $D=D_p(G)$. We want to construct a fundamental domain from D. For every $x\in \overset{\circ}{D}$ we have $d(x,p)<d(gx,p)$ for every $id\neq g\in G$, that is $|Gx\cap D|=1$, where $|H|$ denotes the cardinality of H. However, if $x\in\partial D$ we have $x\in D\cap gD\neq\emptyset$ for some $g\in G$. This implies that $|Gx\cap D|\geq 2$. Yet, for the fundamental region, $B_0$, we need exactly one representative from every orbit. Hence, we need to chose from any orbit $Gx$ intersecting the boundary of D exactly one representative $z\in Gx\cap \partial D$. Let $V$ be a measurable selection of these and finally define $B_0\:=\ \overset{\circ}{D}\cup V.$  By definition, $B_0$ is a fundamental region. Such a measurable selection exists by Theorem 17 and the following Corollary in \cite{dellacherie1975ensembles}.

\begin{example}
 Considering $\R^d$ with group action translations by $\Z^d$ a choice for the fundamental region would be $B_0=[0,1)^d$. If we consider $M=\mathbb{H}^2$ the two dimensional hyperbolic space we can take for G a Fuchsian group acting cocompactly and freely, that is, with no elliptic elements. Then, the closure of the Dirichlet region becomes a hyperbolic polygon (see \cite{katok1992fuchsian}).
\end{example}

\subsection{Optimality}\label{section optimality}

The standard notion of optimality -- minimizers of $\CCost$ or $\Cost$ -- is not well adapted to our setting. For example for any semicoupling $q^\bullet$ between the Lebesgue measure and a Poisson point process of intensity $\beta\leq 1$ we have $\Cost(q^\bullet)=\infty$. Hence, we need to introduce a different notion which we explain in this section.\\
The collection of admissible sets is defined as 
$$\text{Adm}(M)=\{B\in\mathcal B(M): \exists I\subset G, 1\leq|I|<\infty, F \text{ fundamental region} : B=\bigcup_{g\in I}gF\}.$$

For a semicoupling $q^\bullet$ between $\lambda^\bullet$ and $\mu^\bullet$ the \emph{mean transportation cost} of $q^\bullet$ is defined by
$$ \mathfrak C(q^\bullet)\ :=\ \sup_{B\in\text{Adm}(M)}\frac1{m(B)}\EE\left[\int_{M\times B} c(x,y)\ q^\bullet(dx,dy)\right].$$

\begin{definition}
 A semicoupling $q^\bullet$ between $\lambda^\bullet$ and $\mu^\bullet$ is called
\begin{itemize}
 \item[i)] \emph{asymptotically optimal} iff
$$ \mathfrak C(q^\bullet) = \inf_{\tilde q^\bullet\in\Pi_{es}(\lambda^\bullet,\mu^\bullet)}\mathfrak C(\tilde q^\bullet)\ =:\ \mathfrak c_{e,\infty}.$$
\item[ii)] \emph{optimal} iff $q^\bullet$ is equivariant and asymptotically optimal.
\end{itemize}
We will also use several times the quantity
$$\inf_{\tilde q^\bullet\in\Pi_{s}(\lambda^\bullet,\mu^\bullet)}\mathfrak C(\tilde q^\bullet)\ =:\ \mathfrak c_{\infty}.$$
Obviously $\mathfrak c_\infty \leq \mathfrak c_{e,\infty}.$
\end{definition}

Note that the set of optimal semicouplings is convex. This will be useful for the proof of uniqueness.

\begin{remark}
 Equivariant semicouplings $q^\bullet$ are invariant. Hence, they are asymptotically optimal iff  
$$\mathfrak C(q^\bullet)=\EE\left[\int_{M\times B_0} c(x,y) q^\bullet(dx,dy)\right]=\mathfrak c_{e,\infty}.$$
 Because of the invariance, the supremum does not play any role. Moreover, for two different fundamental regions $B_0$ and $\tilde B_0$ define 
$$f(g,h)\ =\ \EE[\CCost(1_{M\times (gB_0\cap h\tilde B_0)}q^\bullet)].$$
Then, for $k\in G$ and equivariant $q^\bullet$ we have $f(g,h)=f(kg,kh).$ Hence, we can apply the mass transport principle to get
$$ \EE\left[\int_{M\times B_0} c(x,y) q^\bullet(dx,dy)\right] = \sum_{h\in G} f(id,h)=\sum_{g\in G} f(g,id) = \EE\left[\int_{M\times \tilde B_0} c(x,y) q^\bullet(dx,dy)\right].$$
Thus, the specific choice of fundamental region is not important for the cost functional $\mathfrak C(\cdot)$ if we restrict to equivariant semicouplings.
\end{remark}

\begin{remark}\label{invariant plus invariant}
The notion of optimality explains why we restrict to stationary probability measures and equivariant random measures. If $\lambda^\bullet$ and $\mu^\bullet$ are just invariant, there does not have to be any invariant semicoupling between them.  Indeed, take $\lambda^\bullet$ a Poisson point process of unit intensity in $\R^d$. It can be written as $\mu^\omega=\sum_{\xi\in\Xi(\omega)}\delta_\xi$. Define $\lambda^\omega:=\sum_{\xi\in\Xi(\omega)}\delta_{-\xi}$ to be the Poisson process that we get if we reflect the first one at the origin. Then $\lambda^\bullet$ and $\mu^\bullet$ are invariant but not jointly invariant, e.g. consider the set $[0,1)^d\times[-1,0)^d$, and not both of them can be equivariant. 
\end{remark}

\subsection{An abstract existence result}
Given that the mean transportation cost is finite the existence of an optimal semicoupling can be shown by an abstract compactness result. A similar reasoning is used to prove Corollary 11 in \cite{holroyd2009geometric}.  

\begin{proposition}\label{abstract exist}
 Let $\lambda^\bullet$ and $\mu^\bullet$ be two equivariant random measures on M with  intensities 1 and $\beta\leq 1$ respectively. Assume that  $\inf_{q^\bullet\in \Pi_{es}(\lambda^\bullet,\mu^\bullet)}\mathfrak C(q^\bullet)=\mathfrak c_{e,\infty}<\infty$, then there exists some equivariant semicoupling $q^\bullet$ between $\lambda^\bullet$ and $\mu^\bullet$ with $\mathfrak C(q^\bullet)=\mathfrak c_{e,\infty}.$
\end{proposition}
\begin{proof}
As $\mathfrak c_{e,\infty}<\infty$ there is a sequence $q^\bullet_n\in\Pi_{is}(\lambda^\bullet,\mu^\bullet) $ such that $\mathfrak C(q^\bullet_n)= c_n\searrow \mathfrak c_{e,\infty}$. Moreover, we can assume that the transportation cost is uniformly bounded by $c_n\leq 2\mathfrak c_{e,\infty} =: c$ for all n. We claim that there is a subsequence $(q^\bullet_{n_k})_{k\in\N}$ of $(q^\bullet_n)_{n\in\N}$ converging to some $q^\bullet\in\Pi_{is}(\lambda^\bullet,\mu^\bullet)$ with $\mathfrak C(q^\bullet)=\mathfrak c_{e,\infty}.$ We prove this in four steps:\\

i) The functional $\mathfrak C(\cdot)$ is lower semicontinuous:\\
It is sufficient to prove that the functional $\CCost(\cdot)$ is lower semicontinuous. Let $(\rho_n)_{n\in\N}$ be any sequence of couplings between finite measures converging to some measure $\rho$ in the vague topology. If $\CCost(\rho_n)=\infty$ for almost all n we are done. Hence, we can assume, that the transportation cost are bounded. Let $(B_0)_r$ denote the r-neighbourhood of $B_0$. For $k\in\R$ let $\phi_k:M\times M\to [0,1]$ be nice cut off functions with $\phi_k(x,y)=1$ on $(B_0)_k \times (B_0)_k$ and $\phi_k(x,y)=0$ if $x\in\complement((B_0)_{k+1})$ or $y\in\complement((B_0)_{k+1})$. Then, we have using continuity of the cost function c(x,y) and by the definition of vague convergence
\begin{eqnarray*}
 \liminf_{n\to\infty}\CCost(\rho_n) &=& \liminf_{n\to\infty} \int_{M\times M} c(x,y) \rho_n(dx,dy)\\
&=& \liminf_{n\to\infty} \sup_{k\in\N} \int_{M\times M} \phi_k(x,y)\ c(x,y) \rho_n(dx,dy)\\
&\geq& \sup_k \liminf_{n\to\infty} \int_{M\times M}\phi_k(x,y)\  c(x,y) \rho_n(dx,dy)\\
&=& \sup_k \int_{M\times M}\phi_k(x,y)\  c(x,y) \rho(dx,dy) = \CCost(\rho).
\end{eqnarray*}
Applying this to $1_{M\times B_0}q^\bullet_n$ shows the lower semicontinuity of $\mathfrak C(\cdot).$

ii) The sequence $(q^\bullet_n)_{n\in\N}$ is tight in $\mathcal M(M\times M\times \Omega)$:\\
Put $f\in C_c(M\times M\times \Omega)$. According to Lemma \ref{vague topology} we have to show $\sup_{n\in\N} q_n^\bullet\P(f)\leq M_f<\infty$ for some constant $M_f$. To this end let $A\subset M$ compact be such that $\supp(f)\subset A\times M\times \Omega.$ We estimate
\begin{eqnarray*}
 \int_{M\times M\times\Omega} f(x,y,\omega) q_n^\omega(dx,dy)\P(d\omega) &\leq &\|f\|_\infty\ \lambda^\bullet\P(A\times\Omega)\\
&\leq& \|f\|_\infty \ m(A) =: M_f.
\end{eqnarray*}
Hence, there is some measure $q^\bullet$ and a subsequence $q^\bullet_{n_k}$ with $q^\bullet_{n_k}\to q^\bullet$ in vague topology on $\mathcal M(M\times M\times \Omega)$. By lower semicontinuity, we have $\mathfrak C(q^\bullet)\leq \liminf \mathfrak C(q^\bullet_{n_k})= \mathfrak c_{e,\infty}.$ Now we have a candidate. We still need to show that it is admissible.\\

iii) $q^\bullet$ is equivariant:\\
Take any continuous compactly supported $f\in C_c(M\times M\times \Omega)$.  By definition of vague convergence
$$ \int f(x,y,\omega) q^\omega_{n_k}(dx,dy)\P(d\omega) \to \int f(x,y,\omega) q^\omega(dx,dy)\P(d\omega).$$
As all the $q^\bullet_{n_k}$ are equivariant, we have for any $g\in G$
\begin{eqnarray*}
 \int f(x,y,\omega) q^\omega_{n_k}(dx,dy)\P(\omega)&=&\int f(g^{-1}x,g^{-1}y,\theta_g\omega) q^{\theta_g\omega}_{n_k}(dx,dy)\P(d\omega)\\
&\to& \int f(g^{-1}x,g^{-1}y,\theta_g\omega) q^{\theta_g\omega}(dx,dy)\P(d\omega).
\end{eqnarray*}

Putting this together, we have for any $g\in G$
$$ \int f(x,y,\omega) q^\omega(dx,dy)\P(d\omega) = \int f(g^{-1}x,g^{-1}y,\theta_g\omega) q^{\theta_g\omega}(dx,dy)\P(d\omega).$$
Hence, $q^\bullet$ is equivariant.\\

iv) $q^\bullet$ is a semicoupling of $\lambda^\bullet$ and $\mu^\bullet$:\\

Fix $h\in C_c(M\times \Omega).$ Put $A\subset M$ compact  such that $\supp(h)\subset A\times \Omega$ and $A\in Adm(M).$ Denote the $R-$neighbourhood of A by $A_R$. By the uniform bound on transportation cost we have
\begin{equation}
q_n^\bullet\P(\complement(A_R),A,\Omega)\leq m(A) \frac{c}{\vartheta(R)},\label{uniform bnd on cplg} 
\end{equation}

uniformly in n. Let $f_R:M\to [0,1]$ be a continuous compactly supported function such that $f_R(x)=1$ for $x\in A_R$ and $f_R(x)=0$ for $x\in\complement A_{R+1}$. As $q_n^\bullet\P$ is a semicoupling of $\lambda^\bullet$ and $\mu^\bullet$ we have due to monotone convergence
\begin{eqnarray*}
\int_{M\times\Omega}h(y,\omega)\mu^\omega(dy)\P(d\omega)&=&\int_{M\times M\times \Omega} h(y,\omega)q_n^\omega(dx,dy)\P(d\omega) \\
&=& \lim_{R\to\infty} \int_{M\times M\times \Omega} f_R(x)h(y,\omega) q_n^\omega(dx,dy)\P(d\omega).
\end{eqnarray*}
Because of the uniform bound (\ref{uniform bnd on cplg}) we have
$$\left|\int_{M\times \Omega} h(x,\omega) \mu^\omega(dx)\P(d\omega)-\int_{M\times M\times \Omega}f_R(x)h(y,\omega) q_{n_k}^\omega(dx,dy)\P(d\omega)\right|\leq m(A) \frac{c\cdot \|h\|_\infty}{\vartheta(R)}.$$
Taking first the limit of $n_k\to\infty$ and then the limit of $R\to\infty$ we conclude using vague convergence and monotone convergence that
\begin{eqnarray*}
0\ &=& \lim_{R\to\infty}\lim_{k\to\infty}\left|\int_{M\times \Omega} h(y,\omega) \mu^\omega(dy)\P(d\omega)-\int_{M\times M\times \Omega}f_R(x)h(y,\omega) q_{n_k}^\omega(dx,dy)\P(d\omega)\right|\\
&=& \lim_{R\to\infty}\left|\int_{M\times \Omega} h(y,\omega) \mu^\omega(dy)\P(d\omega)-\int_{M\times M\times \Omega}f_R(x)h(y,\omega) q^\omega(dx,dy)\P(d\omega)\right|\\
&=& \left|\int_{M\times \Omega} h(y,\omega) \mu^\omega(dy)\P(d\omega)-\int_{M\times M\times \Omega}h(y,\omega) q^\omega(dx,dy)\P(d\omega)\right|
\end{eqnarray*}
This shows that the second marginal equals $\mu^\bullet$. For the first marginal we have for any $k\in C_c(M\times \Omega)$
$$\int_{M\times\Omega} k(x,\omega) q^\omega_{n_k}(dx,dy)\P(d\omega) \leq \int_{M\times\Omega} k(x,\omega)\lambda^\omega(dx)\P(d\omega).$$
In particular, using the function $f_R$ from above we have,
$$\int_{M\times\Omega} f_R(y)\ k(x,\omega) q^\omega_{n_k}(dx,dy)\P(d\omega) \leq \int_{M\times\Omega} k(x,\omega)\lambda^\omega(dx)\P(d\omega).$$
Taking the limit $n_k\to\infty$ yields by vague convergence
$$\int_{M\times\Omega} f_R(y)\ k(x,\omega) q^\omega(dx,dy)\P(d\omega) \leq \int_{M\times\Omega} k(x,\omega)\lambda^\omega(dx)\P(d\omega).$$
Finally taking the supremum over $R$ shows that $q^\bullet$ is indeed a semicoupling of $\lambda^\bullet$ and $\mu^\bullet$.
\end{proof}

\begin{remark}
\begin{itemize}
 \item[i)]   This coupling need not be a factor coupling. We do not know if it is in general true or not that $\mathfrak c_\infty=\mathfrak c_{e,\infty}$, that is, if minimizing the functional $\mathfrak C(\cdot)$ over all semicouplings is the same as minimizing over all equivariant semicouplings. However, in the case that the balls $\Lambda_r\subset G$ are F\o lner sets, we can show equality (see Corollary \ref{lim inf change} and Remark \ref{rem inf eq inf}).
\item[ii)] The same proof shows the existence of optimal semicouplings between $\lambda^\bullet$ and $\mu^\bullet$ with intensities 1 and $\beta\geq 1$ respectively. In this case the ``semi'' is on the side of $\mu^\bullet$ (see also section \ref{s:other}).
\end{itemize}
\end{remark}

\begin{lemma}\label{semicoupling gleich coupling}
 Let $q^\bullet$ be an invariant semicoupling of two random measures $\lambda^\bullet$ and $\mu^\bullet$ with intensities 1 and $\beta\leq 1$ respectively. Then, $q^\bullet$ is a coupling iff $\beta=1$.
\end{lemma}
\begin{proof}
 This is another application of the mass transport principle. Let $B_0$ be a fundamental region and define $f(g,h)=\EE[q^\bullet(gB_0,hB_0)]$. By invariance of $q^\bullet$, we have $f(g,h)=f(kg,kh)$ for any $k\in G$. Hence, we get
$$  1=\EE[\lambda^\bullet(B_0)]\geq \EE[q^\bullet(B_0,M)] = \sum_{g\in G} f(id,g) = \sum_{h\in G} f(h,id) = \EE[q^\bullet(M,B_0)]=\beta.$$
We have equality iff $\beta=1$. By definition of semicoupling, we also have $q^\omega(A,M)\leq \lambda^\omega(A)$ for any $A\subset M$. Hence, in the case of equality we must have $q^\omega(A,M) = \lambda^\omega(A)$ for $\P-$almost all $\omega.$
\end{proof}

\begin{remark}
 The remark above applies again. Considering the case of intensity $\beta\geq 1$ gives that $q^\bullet $ is a coupling iff $\beta=1$.
\end{remark}

\subsection{Assumptions}
Let us summarize the setting and assumptions we work with in the rest of the article.

\begin{itemize}
 \item M will  be a smooth connected non-compact Riemannian manifold with Riemannian volume measure m, such that there is a group G of isometries acting properly discontinuously, cocompactly and freely on M.
\item $B_0$ will denote the chosen fundamental region.
\item $c(x,y)=\vartheta(d(x,y))$ with $\vartheta:\R_+\to\R_+$ such that $\vartheta(0)=0$ and $\lim_{r\to\infty}\vartheta(r)=\infty.$ Given two compactly supported probability measures on M, $\lambda\ll m$ and $\mu$ arbitrary, we will assume that the optimal transportation problem admits a unique solution which is induced by a measurable map T, i.e. $q=(id,T)_*\lambda.$
\item $(\Omega,\mathfrak A,\P)$ will be a  probability space admitting a measurable flow $(\theta_g)_{g\in G}$. $\P$ is assumed to be stationary and $\Omega$ is assumed to be a compact metric space.
\item $\lambda^\bullet$ and $\mu^\bullet$ will be equivariant measure of intensities one respectively $\beta \in (0,\infty).$ Moreover, we assume that $\lambda^\bullet$ is absolutely continuous.

\end{itemize}

\section{Optimal Semicouplings on bounded sets}\label{s:bd semi}
The goal of this section is to prove Theorem \ref{eu:Q+q}, the crucial existence and uniqueness result for optimal semicouplings between $\lambda^\bullet$ and $\mu^\bullet$ restricted to a bounded set. The strategy will be to first prove existence and uniqueness of optimal semicouplings $q=q^\omega$ for deterministic measures $\lambda=\lambda^\omega$ and $\mu=\mu^\omega$. Secondly, we will show that the map $\omega\mapsto q^\omega$ is measurable, which will allow us to deduce Theorem \ref{eu:Q+q}.
\medskip\\  
Optimal semicouplings are solutions of a twofold optimization problem: the optimal choice of a density $\rho\le1$ of the first marginal $\lambda$ and subsequently the optimal choice of a coupling between $\rho\lambda$ and $\mu$.
This twofold optimization problem can also be interpreted as a transport problem with free boundary values.

Throughout this section, we fix the cost function $c(x,y)=\vartheta(d(x,y))$ with
$\vartheta$ -- as before -- being a strictly increasing, continuous function from $\mathbb R_+$ to $\mathbb R_+$ with $\vartheta(0)=0$ and $\lim\limits_{r\to\infty}\vartheta(r)=\infty$. As already mentioned, we additionally assume that the optimal transportation problem between two compactly supported probability measures $\lambda$ and $\mu$ such that $\lambda\ll m$ has a unique solution given by a transportation map, e.g. the optimal coupling is given by $q=(id,T)_*\lambda$. There are very general results on the uniqueness of the solution to the Monge problem for which we refer to chapters 9 and 10 of \cite{villani2009optimal}. To be more concrete we state a uniqueness result for compact manifolds due to McCann \cite{mccann01polar} and  an uniqueness result by Huesmann and Sturm in the simple but for us very interesting case that the measure $\mu$ is discrete.

\begin{theorem}[McCann]\label{u mfd}
Let N be a compact manifold, $\lambda\ll m$ and $\mu$ be probability measures and $c(x,y)=\int_0^{d(x,y)} \tau(s)ds$ with $\tau:\R_+\to\R$ continuously increasing and $\tau(0)=0$. Then, there is a measurable map $T:M\to M\cup \{\eth\}$ such that the unique optimal coupling between $\lambda$ and $\mu$ is given by $q=(id,T)_*\lambda.$
\end{theorem}
The ``cemetery'' $\eth$ in the statement is not really important. This is the place where all points outside of the support of $\lambda$ are sent. We just include it to make some notations easier.

If we assume $\mu$ to be discrete, Lemma 6.1 in \cite{huesmann2010optimal} shows that we can actually take $\vartheta$ to be any continuous strictly increasing function.

\begin{lemma}\label{leb-dirac} Given a finite set $\Xi=\{\xi_1,\ldots,\xi_k\}\subset M$, positive numbers $(a_i)_{1\leq i\leq k}$ summing to one and a probability density $\rho\in L^1(M, m)$. Consider the cost function $c(x,y)=\vartheta(d(x,y))$ for some continuous strictly increasing function $\vartheta:\R_+\to\R_+$ such that $\vartheta(0)=0$ and $\lim_{r\to\infty}\vartheta(r)=\infty.$ If $dim(M)=1$ we exclude the case $\vartheta(r)=r.$

i) There exists a unique coupling $q$ of $\rho\cdot m$ and $\sigma=\sum_{i=1}^k a_i\delta_{\xi_i}$ which minimizes the cost functional $\CCost(\cdot)$.

ii) There exists a ($ m$-a.e. unique) map $T:\{\rho>0\}\to\Xi$ with $T_*(\rho\cdot m)=\sigma$ which minimizes
$\int c(x,T(x))\rho(x)\, m(dx)$.

iii) There exists a ($ m$-a.e. unique) map $T:\{\rho>0\}\to\Xi$ with $T_*(\rho\cdot m)=\sigma$ which is $c$-monotone (in the sense that the closure of $\{(x,T(x)): \ \rho(x)>0\}$ is a
$c$-cyclically monotone set).

iv) The minimizers in (i), (ii) and (iii) are related by $q=(Id,T)_*(\rho\cdot m)$ or, in other words,
$$q(dx,dy)\ =\ \delta_{T(x)}(dy)\, \rho(x)\, m(dx).$$
\end{lemma}

\begin{remark}
 In the case that $dim(M)=1$ and cost function $c(x,y)=d(x,y)$ the optimal coupling between an absolutely continuous measure and a discrete measure need not  be unique. In higher dimensions this is the case, as we get strict inequalities in the triangle inequalities. A counterexample for one dimension is the following. Take $\lambda$ to be the Lebesgue measure on $[0,1]$ and put $\mu=\frac13 \delta_0 + \frac23\delta_{1/16}.$ Then, for any $a\in [1/16,1/3]$
$$q_a(dx,dy)\ =\ 1_{[0,a)}(x)\delta_0(dy)\lambda(dx)+1_{[a,2/3 + a)}(x)\delta_{1/16}(dy)\lambda(dx) + 1_{[a+2/3,1]}(x)\delta_0(dy)\lambda(dx)$$
is an optimal coupling of $\lambda$ and $\mu$ with $\CCost(q_a)= 11/24$.
\end{remark}

\begin{remark}\label{shape of optimal transport map}
 In the case of cost function $c(x,y)=\frac1p d^p(x,y)$ the optimal transportation map is given by 
$$ T(x)\ =\ \exp\left(d(x,\xi_j)\frac{\nabla \Phi_j(x)}{|\nabla \Phi_j(x)|}\right) $$
for functions $\Phi_i(z)=-\frac{1}{p}d^p(z,\xi_i)+b_i$ with constants $b_i$ and $j$ such that $\Phi_j(x)=\max_{1\leq i\leq k} \Phi_i(x)$ (e.g. see \cite{mccann01polar}).
\end{remark}

Given two deterministic measures $\lambda=f\cdot m$ for some compactly supported density $f$ (in particular $\lambda\ll m$) and an arbitrary finite measure $\mu$ with $\supp(\mu)\subset A$ for some compact set $A$ such that $\mu(M)\leq \lambda(M)<\infty$. We are looking for minimizers of 
$$\CCost(q)=\int c(x,y)q(dx,dy)$$
under all semicouplings $q$ of $\lambda$ and $\mu$. The key step is a nice observation by Figalli, namely Proposition 2.4 in \cite{Figalli2010optimal}. The version we state here is adapted to our setting.
\begin{proposition}[Figalli]\label{fig obs}
 Let q be a $\CCost$ minimizing semicoupling between $\lambda$ and $\mu$. Write $f_q\cdot m=(\pi_1)_*q$. Consider the Monge-Kantorovich problem:
$$\text{minimize } C(\gamma)\ =\ \int_{M\times M} c(x,y) \gamma(dx,dy)$$
among all $\gamma$ which have $\lambda$ and $\mu + (f-f_q)\cdot m$ as first and second marginals, respectively. Then, the unique minimizer is given by
$$ q + (id\times id)_*(f-f_q)\cdot m.$$
\end{proposition}
This allows us to show that all minimizers of $\CCost$ are concentrated on the same graph which also gives us uniqueness:
\begin{proposition}\label{u sc bd set}
 There is a unique $\CCost$ minimizing semicoupling between $\lambda$ and $\mu$. It is given as $q = (id,T)_*(\rho\cdot\lambda)$ for some measurable map $T:M\to M\cup\{\eth\}$ and density $\rho$.
\end{proposition}
\begin{proof}
  (i) The functional $\CCost(\cdot)$ is lower semicontinuous on $\mathcal M(M\times M)$ wrt weak convergence of measures. Indeed, take a sequence of measures $(q_n)_{n\in\N}$ converging weakly to some q. Then we have by continuity of the cost function $c(\cdot,\cdot)$:
\begin{eqnarray*}
 \int c(x,y)\ q(dx,dy) &=& \sup_{k\in\N} \int c(x,y) \wedge k\ q(dx,dy)\\
 &=& \sup_{k\in\N} \lim_{n\to\infty} \int c(x,y) \wedge k\ q_n(dx,dy)\\
& \leq& \liminf_{n\to\infty}\int c(x,y)\ q_n(dx,dy).
\end{eqnarray*}
(ii) Let $\mathcal O$ denote the set of all semicouplings of $\lambda$ and $\mu$ and $\mathcal O_1$ denote the set of all semicouplings q satisfying $\CCost(q)\leq 2\inf_{q\in\mathcal O}\CCost(q)=:2c$. Then $\mathcal O_1$ is relatively compact wrt weak topology. Indeed, $q(M\times \complement A)=0$ for all $q\in \mathcal O_1$ and
$$q(\complement (A_r)\times A)\le\frac1{\vartheta\left(r\right)}\cdot \CCost(q) \le\frac2{\vartheta\left(r\right)} c$$
for each $r>0$ where $A_r$ denotes the closed $r$-neighborhood of $A$ in $M$. Thus, for any $\epsilon>0$ there exists a compact set $K={A_r}\times A$ in $M\times M$ such that $q(\complement K)\le\epsilon$ uniformly in $q\in \mathcal O_1$.\\
(iii) The set $\mathcal O$ is closed wrt weak topology. Indeed, if $q_n\to q$ then $(\pi_1)_*q_n\to (\pi_1)_*q$ and $(\pi_2)_*q_n\to (\pi_2)_*q$. Hence $\mathcal O_1$ is compact and $\CCost$ attains its minimum on $\mathcal O$. Let q denote one such minimizer. Its first marginal is absolutely continuous to m. By Theorem \ref{u mfd} and Lemma \ref{leb-dirac}, there is a measurable map $T:M\to M\cup\{\eth\}$ and  densities $\tilde f_q, f_q$ such that $q=(id,T)_*(\tilde f_q\cdot\lambda)=(id,T)_*( f_q\cdot m)$.\\
(iv) Given a minimizer of $\CCost$, say q. By Proposition \ref{fig obs}, $\tilde q:= q + (id,id)_*(f-f_q)\cdot m$ solves 
$$\min C(\gamma)=\int c(x,y) \gamma(dx,dy)$$
under all $\gamma$ which have $\lambda$ and $\mu + (f-f_q) m$ as first respectively second marginals, where $f_q \cdot m=(\pi_1)_*q$ as above. By Theorem \ref{u mfd} and Lemma \ref{leb-dirac}, there is a measurable map $S$ such that $\tilde q=(id,S)_*\lambda$. That is, $\tilde q$ and in particular q are concentrated on the graph of $S$. By definition $\tilde q= q + (id,id)_*(f-f_q)\cdot m$ and, therefore, we must have $S(x)=x$ on $\{f>f_q\}$.\\
(v) This finally allows us to deduce uniqueness. By the previous step, we know that any convex combination of optimal semicouplings is concentrated on a graph. This implies that all optimal semicouplings are concentrated on the \emph{same} graph. Moreover, Proposition \ref{fig obs} implies that if we do not transport all the $\lambda$ mass in one point we leave it where it is. Hence, all optimal semicouplings choose the same density $\rho$ of $\lambda$ and therefore coincide.\\
Assume there are two optimal semicouplings $q_1$ and $q_2$. Then $q_3:=\frac12 (q_1+q_2)$ is optimal as well. By the previous step for any $i\in\{1,2,3\}$, we get  maps $S_i$ such that $q_i$ is concentrated on the graph of $S_i$. Moreover, we have $S_3(x)=x$ on the set $\{f>f_{q_3}\}=\{f>f_{q_1}\}\cup\{f>f_{q_2}\}$, where again $f_{q_i}\cdot m=(\pi_1)_* q_i$. As $q_3$ is concentrated on the graph of $S_3$, $q_1$ and $q_2$ must be concentrated on the same graph. Hence, we have $S_3=S_i$ on $\{f_{q_i}>0\}$ for $ i=1,2.$ We also know from the previous step that $S_i(x)=x$ on $\{f>f_{q_i}\}\subset\{f>f_{q_3}\}$. This gives, that $S_3=S_1=S_2$ on $\{f>0\}$.\\
We still need to show that $\{f_{q_1}>0\}=\{f_{q_2}>0\}$. Put $A_1 := \{f_{q_1}>f_{q_2}\}$ and $A_2 := \{f_{q_2}>f_{q_1}\}$ and assume $m(A_1)>0$. As $A_1\subset\{f>f_{q_2}\}$ we know that $S_3(x)=x$ on $A_1$ and similarly $S_3(x)=x$ on $A_2$. Now consider 
$$A\ :=\ S_3^{-1}(A_1)\ =\ (A\cap \{f_{q_1}=f_{q_2}\})\cup (A\cap A_1) \cup (A\cap A_2).$$
As $S_3(A_2)\subset A_2$ and $A_1\cap A_2=\emptyset$ we have $A\cap A_2=\emptyset$. Therefore, we can conclude
\begin{eqnarray*}
\mu(A_1)&=&(S_3)_*f_{q_1}m(A_1)\ =\ f_{q_1}m(A_1)+f_{q_1}m(A\cap\{f_{q_1}=f_{q_2}\})\\
&>& f_{q_2}m(A_1) + f_{q_2}m(A\cap\{f_{q_1}=f_{q_2}\})\\ 
&=& (S_3)_*f_{q_2}m(A_1)\ =\ \mu(A_1),
\end{eqnarray*}
which is a contradiction, proving $q_1=q_2$.
\end{proof}

\begin{remark}
 Let $q=(id,T)_*(\rho\lambda)$ be the optimal semicoupling of $\lambda$ and $\mu$. If $\mu$ happens to be discrete, we have $\rho(x)\in\{0,1\}$ m almost everywhere. Indeed, assume the contrary. Then, there is $\xi\in \supp(\mu)$ such that on $ U:= T^{-1}(\xi)$ we have  $\rho\in(0,1)$ on some set of positive $\lambda$ measure. Let $R$ be such that $\lambda(U\cap B(\xi,R))= \mu(\{\xi\}),$ where $B(\xi,R)$ denotes the ball of radius $R$ around $\xi$. Put $V=U\cap B(\xi,R)$ and 
$$\tilde q(dx,dy)\ =\ q(dx,dy) - 1_U(x)\rho(x)\delta_{\xi}(dy)\lambda(dx) + 1_V(x)\delta_{\xi}(dy)\lambda(dx).$$
This means, we take the same transportation map, but use the $\lambda$ mass more efficiently. $\tilde q$ leaves some $\lambda$ mass far out and instead uses the same amount of $\lambda$ mass which is closer to the target $\xi$. By construction, we have $\CCost(q)>\CCost(\tilde q)$ contradicting optimality of q.
\end{remark}

We  showed the existence and uniqueness of optimal semicouplings between deterministic measures. The next step in the proof of Theorem \ref{eu:Q+q} is to show the measurability of the mapping $\omega\mapsto \Phi(\lambda^\omega,1_A\mu^\omega)=q_A^\omega$ the unique optimal semicoupling between $\lambda^\omega$ and $1_A\mu^\omega$. The mapping $\omega\mapsto (\lambda^\omega,1_A\mu^\omega)$ is measurable by definition. Hence, we have to show that $(\lambda^\omega,1_A\mu^\omega)\mapsto \Phi(\lambda^\omega,1_A\mu^\omega)$ is measurable. We will show a bit more, namely that this mapping is actually continuous. We start with a simple but important observation about optimal semicouplings.
\medskip\\
Denote the one-point compactification of M by $M\cup\{\eth\}$ and let $\tilde\vartheta(r)$ be such that it is equal to $\vartheta(r)$ on a very large box, say $[0,K]$ and then tends continuously to zero such that $\tilde c(x,\eth)=\tilde\vartheta(d(x,\eth))=\lim_{r\to\infty}\tilde\vartheta(r)=0$ for any $x\in M$. By a slight abuse of notation, we also write $\eth:M\to \{\eth\}$ for the map $x\mapsto \eth.$

\begin{lemma}
Let two measures $\lambda$ and $\mu$ on M be given such that $\infty>\lambda(M)=N\geq\mu(M)=\alpha$ and assume there is a ball $B(x,K/2)$ such that $\supp(\lambda),\supp(\mu)\subset B(x,K/2)$. Then, q is an optimal semicoupling between $\lambda$ and $\mu$ wrt to the cost function $c(\cdot,\cdot)$ iff $\tilde q=q + (id,\eth)_*(1-f_q)\cdot\lambda$ is an optimal \emph{coupling} between $\lambda$ and $\tilde \mu=\mu+(N-\alpha)\delta_\eth$ wrt the cost function $\tilde c(\cdot,\cdot),$ where $(\pi_1)_*q=f_q\lambda.$
\end{lemma}
\begin{proof}
Let q be any semicoupling between $\lambda$ and $\mu$. Then $\tilde q=q + (id,\eth)_*(1-f_q)\cdot\lambda$ defines a coupling between $\lambda$ and $\tilde \mu$. Moreover, the transportation cost of the semicoupling and the one of the coupling are exactly the same, that is $\CCost(q)=\CCost(\tilde q).$ Hence, q is optimal iff $\tilde q$ is optimal.
\end{proof}

This allows to deduce the continuity of $\Phi$ from the classical theory of optimal transportation.

\begin{lemma}\label{continuity bd sc}
 Given a sequence of measures $(\lambda_n)_{n\in\N}$ converging vaguely to some $\lambda$, all absolutely continuous to m with $\lambda_n(M)=\lambda(M)=\infty.$ Moreover, let $(\mu_n)_{n\in\N}$ be a sequence of finite measures converging weakly to some finite measure $\mu$, all concentrated on some bounded set $A\subset M$. Let $q_n$ be the optimal semicoupling between $\lambda_n$ and $\mu_n$ and $q$ be the optimal semicoupling between $\lambda$ and $\mu$. Then, $q_n$ converges weakly to q. In particular, the map $(\lambda,\mu)\mapsto \Phi(\lambda,\mu)=q$ is continuous.
\end{lemma}
\begin{proof}
 i)\ As $(\mu_n)_{n\in\N}$ converge to $\mu$ and $\mu$ is finite, we can assume that $\sup_n\mu_n(M),\mu(M)\leq\alpha<\infty.$ As $\lambda_n$ and $\lambda$ have infinite mass for any $x\in M$ and $k\in \R$ there is a radius $R(x,k)<\infty$ such that $\lambda(B(x,R(x,k)))\geq k$, where $B(x,R)$ denotes the closed ball around x of radius R. Fix an arbitrary $x\in A$ and set $R_1=R(x,\alpha)+\mbox{diam}(A)$ and $R_2=R(x,2\alpha)+\mbox{diam}(A).$ Because $(\lambda_n)_{n\in\N}$ converge to $\lambda$ we can assume that for any n $\lambda_n(B(x,R_2))\geq \lambda(B(x,R_1)=N> \alpha.$

\medskip

ii)\ Optimality of $q_n$ and $q$ implies that $\supp(q_n)\subset B(x,R_2)\times A$ and $\supp(q)\subset B(x,R_1)\times A$. Because otherwise there is still some mass lying closer to the target than the mass which is transported into the target. For any n let $r_n\leq R_2$ be such that $\lambda_n(B(x,r_n))=N$. Such choices exist as $\lambda_n\ll m$ for all n. Then, we even know that $\supp(q_n)\subset B(x,r_n)\times A$. Set $\tilde \lambda_n=1_{B(x,r_n)}\lambda_n$ and $\tilde\lambda=1_{B(x,R_1)}\lambda$. Then the optimal semicoupling between $\lambda_n$ and $\mu_n$ is the same as the optimal semicoupling between $\tilde\lambda_n$ and $\mu_n$ and similarly the optimal semicoupling between $\lambda$ and $\mu$ is the same as the optimal semicoupling between $\tilde\lambda$ and $\mu$. Moreover, because for any n $\tilde \lambda_n$ is compactly supported with total mass N the vague convergence $\lambda_n\to\lambda$ implies weak convergence of $\tilde \lambda_n\to\tilde\lambda.$

\medskip

iii)\ Now we are in a setting where we can apply the previous Lemma. Set $K=2R_2$ and define $\tilde\vartheta, \tilde\mu_n,\tilde\mu$ as above. Then $\tilde q_n$ and $\tilde q$ are optimal couplings between $\tilde\lambda_n$ and $\tilde\mu_n$ and $\tilde\lambda$ and $\tilde\mu$ respectively wrt to the cost function $\tilde c(\cdot,\cdot)$. The cost function $\tilde c$ is continuous and $M$ and $M\cup\{\eth\}$ are Polish spaces. Hence, we can apply the stability result of the classical optimal transportation theory (e.g. Theorem 5.20 in \cite{villani2009optimal}) to conclude that $\tilde q_n\to\tilde q$ weakly and therefore $q_n\to q$ weakly.
\end{proof}

Take a pair of equivariant random measure $(\lambda^\bullet,\mu^\bullet)$ with $\lambda^\omega\ll m$ as usual. For a given $\omega\in\Omega$ we want to apply the results of the previous Lemma to a fixed realization $(\lambda^\omega,\mu^\omega)$. Then, for any bounded Borel set $A\subset M$, there is a unique optimal semicoupling $q_A^\omega$ between $\lambda^\omega$ and $1_A\mu^\omega,$ that is, a unique minimizer of the cost function $\CCost$ among all semicouplings of $\lambda^\omega$ and $1_A\mu^\omega$.

\begin{lemma}\label{qA measurable}
 For each bounded Borel set $A\subset M$ the map $\omega\mapsto q^\omega_A$ is measurable.
\end{lemma}
\begin{proof}
 We saw that the map $\Phi:(\lambda^\omega, 1_A\mu^\omega)=q_A^\omega$ is continuous. By definition of random measures the map $\omega\mapsto (\lambda^\omega, 1_A\mu^\omega)$ is measurable. Hence, the map
$$\omega \mapsto \Phi(\lambda^\omega,1_A\mu^\omega)=q_A^\omega$$
is measurable.
\end{proof}

The uniqueness and measurably of $q^\omega_A$ allows us to finally deduce

\begin{theorem}
 (i) For each bounded Borel set $A\subset M$ there exists a unique semicoupling $\Q_A$ of $\lambda^\bullet\P$ and $(1_A\mu^\bullet)\mathbb P$ which minimizes the mean cost functional $\Cost(.)$.

(ii) The measure  $\Q_A$ can be disintegrated as $\Q_A(dx,dy,d\omega):=q_A^\omega(dx,dy)\,{\mathbb P}(d\omega)$ where for $\mathbb P$-a.e. $\omega$ the measure $q_A^\omega$ is the unique minimizer of the cost functional $\CCost(.)$ among the semicouplings of $\lambda^\omega$ and $1_A\mu^\omega$.

(iii) $\Cost(\Q_A)=\int_\Omega\CCost(q_A^\omega)\, \mathbb P(d\omega).$
\end{theorem}
\begin{proof}
The existence of a minimizer is proven along the same lines as in the previous proposition: We choose an approximating sequences
$\Q_n$ in $\mathcal M(M\times M\times\Omega)$ -- instead of  a sequence $q_n$ in $\mathcal M(M \times M)$ --
minimizing the lower semicontinuous functional $\Cost(\cdot)$.
Existence of a limit follows as before from tightness of the set of all semicouplings $\Q$ with $\Cost(\Q)\le 2\inf_{\tilde\Q}\Cost(\tilde\Q)$.

For each semicoupling $\Q$ of $\lambda^\bullet$ and $1_A\mu^\bullet$ with disintegration as $q^\bullet\mathbb P$ we obviously have
$$\Cost(\Q)=\int_\Omega\CCost(q^\omega)\, d\mathbb P(\omega).$$
Hence, $\Q$ is a minimizer of the functional $\Cost(\cdot)$ (among all semicouplings of  $\lambda^\bullet$ and $1_A\mu^\bullet$) if and only if for $\mathbb P$-a.e. $\omega\in\Omega$ the measure $q^\omega$ is a minimizer of the functional $\CCost(.)$ (among all semicouplings of $\lambda^\omega$ and $1_A\mu^\omega$).

Uniqueness of the minimizer of $\CCost(\cdot)$ therefore implies uniqueness of the minimizer of $\Cost(\cdot)$.
\end{proof}

\begin{corollary}\label{invariance of Q}
The optimal semicouplings $Q_A=q_A^\bullet\P$ are equivariant in the sense that
$$ Q_{gA}(gC,gD,\theta_g\omega)= Q_A(C,D,\omega),$$
for any $g\in G$ and  $C,D\in\mathcal B(M).$
\end{corollary}
\begin{proof}
This is a consequence of the equivariance of $\lambda^\bullet$ and $\mu^\bullet$ and the fact that $q_A^\omega$ is a deterministic function of $\lambda^\omega$ and $1_A\mu^\omega$.
\end{proof}

\section{Uniqueness}\label{s:u}

The aim of this section is to prove  Theorem \ref{main thm 1}, the uniqueness of optimal semicouplings. Moreover, the representation of optimal semicouplings, that we get as a byproduct of the  uniqueness statement, allows to draw several conclusions about the geometry of the cells of the induced allocations.

Throughout this section we fix two equivariant random measures $\lambda^\bullet$ and $\mu^\bullet$ of unit resp. subunit intensities on M with finite optimal mean transportation cost $\mathfrak c_{e,\infty}$. Moreover, we assume that $\lambda^\bullet$ is absolutely continuous.

\begin{proposition}\label{loc-opt}
 Given a semicoupling $q^\omega$ of $\lambda^\omega$ and $\mu^\omega$ for fixed $\omega\in\Omega$, then the following properties are equivalent.
\begin{enumerate}
\item For each bounded Borel set $A\subset M$, the measure $1_{M\times A}q^\omega$ is the unique optimal coupling of the measures $\lambda_A^\omega(\cdot):=q^\omega(\cdot,A)$ and $1_A\mu^\omega$.
\item The support of $q^\omega$ is $c$-cyclically monotone, more precisely,
$$\sum_{i=1}^N c(x_i,y_i)\le \sum_{i=1}^N c(x_i,y_{i+1})$$
for any $N\in\N$ and any choice of points $(x_1, y_1), \ldots , (x_N, y_N)$ in $\mathrm{supp}(q^\omega)$
with the convention $y_{N+1} = y_1$.
\item There exists a nonnegative density $\rho^\omega$ and a $c$-cyclically monotone map $T^\omega: \{\rho^\omega>0\}\to M$ such that
\begin{equation}\label{trans-map}
q^\omega=\left(Id,T^\omega\right)_*(\rho^\omega\,\lambda^\omega).
\end{equation}
Recall that, by definition,  a map $T$ is $c$-cyclically monotone iff the closure of its graph $\{(x,T(x)):\ x\in \{\rho^\omega>0\} \}$ is a $c$-cyclically monotone set.
\end{enumerate}
\end{proposition}
\begin{proof}
 $(iii)\Rightarrow(ii)\Rightarrow(i)$ follows from Theorem \ref{u mfd} and Lemma \ref{leb-dirac}.  \\
\smallskip

$(i)\Rightarrow (iii):$ Take a nested sequence of convex sets $(K_n)_n$ such that $K_n\nearrow M$. By assumption $1_{M\times K_n}q^\omega$ is the unique optimal coupling between $\lambda_{K_n}^\omega\ll m$ and $1_{K_n}\mu^\omega$. By Proposition \ref{u sc bd set} or Theorem \ref{u mfd}, there exists a density $\rho_n^\omega$ and a map $T_n^\omega:\{\rho_n^\omega>0\}\to M$ such that $1_{M\times K_n}q^\omega=(id,T_n^\omega)_* (\rho_n^\omega \lambda_{K_n}^\omega)$. Set $A_n^\omega=\{\rho_n^\omega>0\}$. As $K_n\subset K_{n+1}$ we have $A_n^\omega\subset A_{n+1}^\omega$. Subtransports of optimal transports are optimal again. Therefore, we have 
$$T_{n+1}^\omega =T_n^\omega \ \text{ on } A_n^\omega$$
 implying $\rho^\omega_{n+1} =\rho^\omega_n\ \text{ on } A_n^\omega.$ Hence, the limits 
$$T^\omega=\lim_n T_n^\omega,\ A^\omega=\lim_n A_n^\omega\ \text{ and }\ \rho^\omega=\lim_n\rho_n^\omega$$
 exist and define a c-cyclically monotone map $T^\omega:A^\omega\to M$ such that on $A^\omega \times M:$
\begin{equation*}
q^\omega=(id,T^\omega)_*(\rho^\omega\lambda^\omega).
\end{equation*}
\end{proof}

\begin{figure*}
\begin{center}
\includegraphics[scale=0.6]{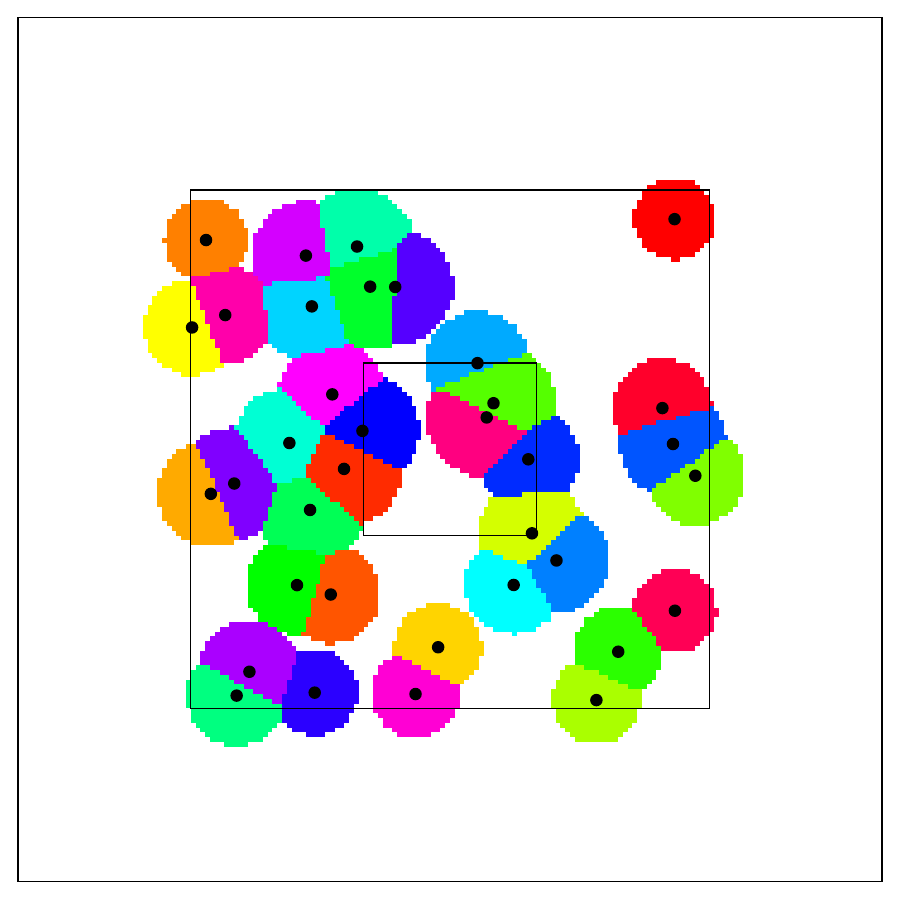} 
\includegraphics[scale=0.6]{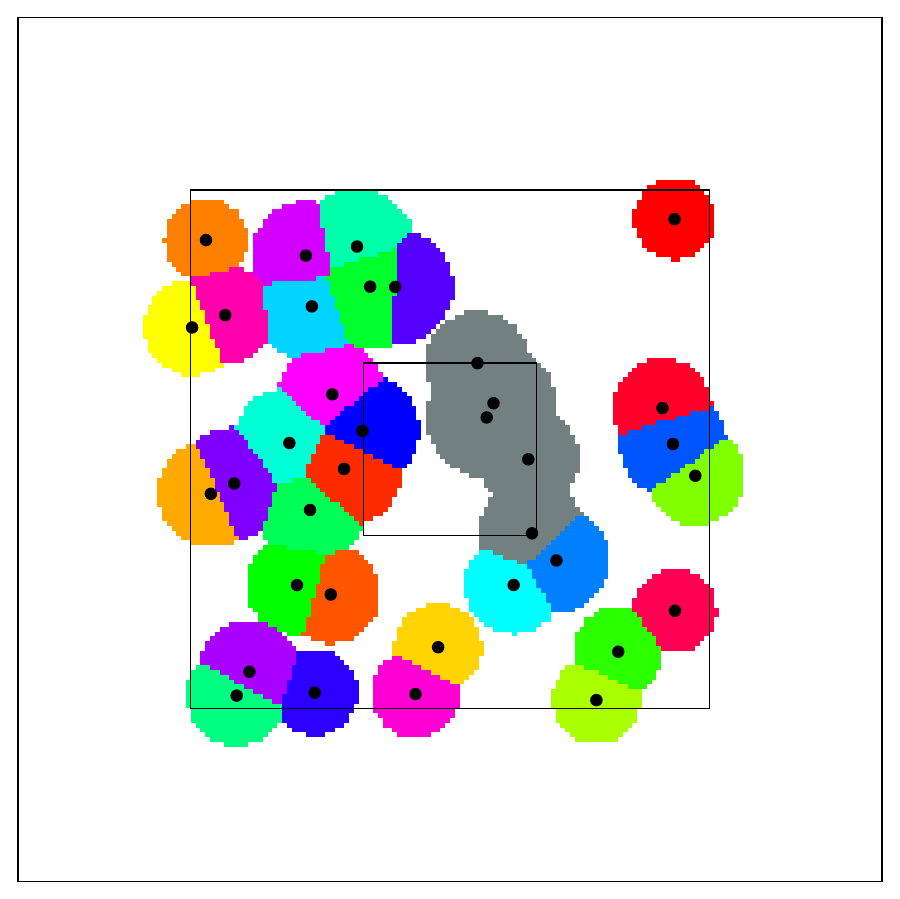}
\end{center}
\caption{The left picture is a semicoupling of Lebesgue and 36 points with cost function $c(x,y)=|x-y|^4$. In the right picture, the five points within the small cube can choose new partners from the mass that was transported to them in the left picture (corresponding to the measure $\lambda_A^\omega$). If the semicoupling on the left hand side is locally optimal, then the points in the small cube on the right hand side will choose from the gray region exactly the partners they have in the left picture. }\label{figure 4}
\end{figure*}

\begin{remark} In the sequel, any \emph{transport map} $T^\omega:A^\omega\to M$ as above
will be extended to a map $T^\omega:M\to M\cup\{\eth\}$  by putting
$T^\omega(x):=\eth$ for all $x\in M\setminus A^\omega$ where $\eth$ denotes an isolated point added to $M$ ('point at infinity', 'cemetery'). Then (\ref{trans-map}) reads
\begin{equation}\label{trans-map2}
q^\omega=\left(Id,T^\omega\right)_*\rho^\omega\lambda^\omega \quad\mbox{on }M\times M.
\end{equation}
Moreover, we put $c(x,T^\omega(x))\ =\ c(x,\eth)\ :=\ 0$ for $x\in M\setminus A^\omega$. If we know a priori that $\rho^\omega(x)\in\{0,1\}$ almost surely (\ref{trans-map2}) simplifies to
\begin{equation}\label{trans-map3}
q^\omega=\left(Id,T^\omega\right)_*\lambda^\omega \quad\mbox{on }M\times M.
\end{equation}
\end{remark}

\begin{definition}\label{def loc-opt}
 A semicoupling $Q=q^\bullet\mathbb P$ of $\lambda^\bullet$ and $\mu^\bullet$ is called \emph{locally optimal} iff  some (hence every)  property  of the previous proposition is satisfied for $\mathbb P$-a.e. $\omega\in\Omega$.
\end{definition}

\begin{remark}\ \label{remark on loc opt}
(i) \quad
Asymptotic optimality is not sufficient for uniqueness and it does not imply local optimality: Consider the Lebesgue measure $\lambda^\bullet = \leb = \lambda$ and a Poisson point process $\mu^\bullet$ of unit intensity on $\R^d$. Let us fix the cost function $c(x,y)=|x-y|^2.$ Lemma \ref{super} shows that $\mathfrak c_0$ the optimal mean transportation cost on the unit cube, that is the cost of the optimal semicoupling of $\lambda$ and  $1_{[0,1)^d}\mu^\bullet,$ is strictly less than the optimal mean transportation cost on a big cube, say $[0,10^{10})^d$. Moreover, for any semicoupling $q^\bullet$ between $\lambda$ and $\mu^\bullet$  Lemma \ref{super} implies that
$$\mathfrak C(q^\bullet)=\liminf_{n\to\infty} \frac1{\lambda(B_n)} \EE\left[\CCost(1_{\R^d\times B_n}q^\bullet)\right],$$
for $B_n=[-2^{n-1},2^{n-1})^d$. In other words, it is more costly to transport in one big cube than in many small cubes separately. For all $n\in\N$, let $\rho_n:\Omega\times\R^d\to [0,1]$ be the unique optimal density for the transport problem between $\lambda$ and $1_{B_n}\mu^\bullet$, that is the optimal semicoupling is given by $q_n^\bullet=(id,T_n^\bullet)_*(\rho_n^\bullet\lambda).$ Let $\kappa_n^\bullet$ be the following semicoupling between $\lambda$ and $1_{B_n}\mu^\bullet$
$$\kappa^\omega(dx,dy)\ =\ q_0^\omega(dx,dy) + \sigma_n^\omega(dx,dy),$$
where $\sigma^\omega_n$ is the unique optimal \emph{coupling} between $(\rho_n^\omega-\rho_0^\omega)\cdot\lambda$ and $1_{B_n\setminus B_0}\mu^\omega.$ Let $f_n:\Omega\times\R^d\to[0,1]$ be such that $1_{R^d\times (B_n\setminus B_0)}q_n^\omega=(id,T^\omega_n)_*(f_n^\omega\lambda).$ Denote by $\W_2$  the expectation of the usual $L^2-$ Wasserstein distance. Then, we can estimate using the triangle inequality
\begin{eqnarray*}
\Cost^{1/2}(\kappa^\bullet_n)\ &=&\ \Cost^{1/2}(q_0^\bullet)+ \W_2((\rho_n-\rho_0)\cdot\lambda, 1_{(B_n\setminus B_0)}\mu^\bullet)  \\
& \leq& \ \Cost^{1/2}(q_0^\bullet) + \Cost^{1/2}(q^\bullet_n) + \W_2((\rho_n-\rho_0)\cdot\lambda, f_n\cdot \lambda).
\end{eqnarray*}
 Set $Z_l=\mu^\bullet(B_l)$. Note that $(\rho_n^\omega-\rho_0^\omega)$ and $f_n^\omega$ coincide on a set of Lebesgue measure of mass at least $Z_n^\omega-Z_0^\omega$. This allows to estimate $\W_2((\rho_n-\rho_0)\cdot\lambda, f_n\cdot \lambda)$ very roughly from above (for similar less rough and much more detailed estimates we refer to  the cost estimates in section 5 of \cite{huesmann2010optimal}). We have to transport mass of amount at most $Z_0$ at most a distance $R_1=2h\cdot Z_n^{1/d}+2\sqrt{d}2^n$ for some constant $h$, e.g. $h=(\Gamma(\frac{d}{2} +1))^{1/d}$ would do. Indeed, $\rho_n^\omega$ must be supported in a $h\cdot Z^{1/d}$ neighbourhood of $B_n$ because we could otherwise produce a cheaper semicoupling (see Lemma \ref{continuity bd sc}). This gives using the estimates on Poisson moments of Lemma 5.11 in \cite{huesmann2010optimal}
\begin{eqnarray*}
&&\W^2_2((\rho_n-\rho_0)\cdot\lambda, f_n\cdot \lambda) \leq \EE\left[R_1^2\cdot Z_0\right]\\
&\leq& C_1\ \left(\EE[Z_0^2]^{1/2}\EE[Z_n^{4/d}]^{1/2}+\lambda(B_n)^{2/d}+2\cdot\EE[Z_0^2]^{1/2}\lambda(B_n)^{1/d}\EE[Z_n^{2/d}]^{1/2}\right)\\
& \leq& C_2\ \lambda(B_n)^{2/d}, 
\end{eqnarray*}
for some constants $C_1$ and $C_2$. In particular, if we take $d\geq 3$ this shows that
$$\liminf_{n\to\infty} \frac1{\lambda(B_n)}\Cost(\kappa^\bullet_n)=\liminf_{n\to\infty} \frac1{\lambda(B_n)}\Cost(q^\bullet_n).$$
Hence, as in Proposition \ref{abstract exist} we can show that $\kappa_n^\bullet$ converges along a subsequence to some semicoupling $\kappa^\bullet$ between $\lambda$ and $\mu^\bullet$ which is asymptotically optimal but not locally optimal.
\medskip
 
(ii) \quad Local optimality does not imply asymptotic optimality and it is not sufficient for uniqueness: For instance in the case $M=\R^d, c(x,y)=|x-y|^2$, given any coupling $q^\bullet$ of $\leb$ and a Poisson point process $\mu^\bullet$ and $z\in\R^d\setminus\{0\}$ then
$$\tilde q^\omega(dx,dy):=q^\omega(d(x+z),dy)$$
defines another locally optimal coupling  of $\leb$ and $\mu^\bullet$. Not all of them can be asymptotically optimal.

\medskip

(iii) \quad The name \emph{local optimality} might be misleading in the context of semicouplings. Consider a Poisson process $\mu^\bullet$ of intensity 1/2 and let $q^\bullet$ be an optimal coupling between $1/2 \leb$ and $\mu^\bullet$. Then, it is locally optimal (see Theorem \ref{opt loc opt}) according to this definition. However, as we left half of the Lebesgue measure laying around we can everywhere locally produce a coupling with less cost. In short, the optimality does not refer to the choice of density only to the use of the chosen density.

\medskip

(iv)\quad
Note that local optimality --- in contrast to asymptotic optimality and equivariance --- is not preserved under convex combinations. It is an open question if local optimality and asymptotic optimality imply uniqueness.
\end{remark}

Given two random measures $\gamma^\bullet, \eta^\bullet: \Omega\to\mathcal M(M)$ with $\gamma^\omega(M) = \eta^\omega(M)<\infty$ for all $\omega\in\Omega$
we define the  \emph{optimal mean transportation cost} by
$$\Cost(\gamma^\bullet,\eta^\bullet)\ :=\ \inf\left\{\Cost(q^\bullet): \ q^\omega\in\Pi(\gamma^\omega,\eta^\omega) \ \mbox{for a.e. }\omega \in \Omega\right\}.$$

Given a semicoupling $q^\bullet$ of $\lambda^\bullet$ and $\mu^\bullet$ and  a bounded Borel set $A\subset M$ recall the definition of $\lambda_A^\bullet$ from Prop. \ref{loc-opt}. We define
the \emph{efficiency of the semicoupling $q^\bullet$ on the set $A$} by
$$\mathfrak {eff}_{A}(q^\bullet):=\frac{\Cost(\lambda^\bullet_A\, ,\,1_A\mu^\bullet)}{\Cost(1_{M\times A}q^\bullet)}.$$
It is a number in $(0,1]$. The semicoupling $q^\bullet$ is said to be efficient on $A$ iff $\mathfrak {eff}_{A}(q^\bullet)=1$. Otherwise, it is inefficient on $A$. As noted in the remark above in the case of true \emph{semi}couplings this notion might mislead the intuition.

\begin{lemma} 
(i) $q^\bullet$ is locally optimal if and only if $\mathfrak {eff}_{A}(q^\bullet)=1$ for all bounded Borel sets  $A\subset M$.

(ii) $\mathfrak {eff}_{A}(q^\bullet)=1$ for some $A\subset M$ implies $\mathfrak {eff}_{A'}(q^\bullet)=1$ for all $ A'\subset A$, where we set $0/0=1$. 
\end{lemma}

\begin{proof} (i) Let $A$ be given and  $\omega\in\Omega$ be fixed. Then
 $1_{M \times A}q^\omega$ is the optimal semicoupling of the measures $\lambda^\omega_A$ and $1_A\mu^\omega$ if and only if
\begin{equation}\label{eff-path}
\CCost(1_{M \times A}q^\omega)=\CCost\left(\lambda_A^\omega\, ,\,1_A\mu^\omega\right).
\end{equation}
On the other hand, $\mathfrak {eff}_{A}(q^\bullet)=1$ is equivalent to
$$\EE\left[\CCost(1_{M \times A}q^\bullet)\right]=\EE\left[\CCost\left(\lambda_A^\bullet\, ,\,1_A\mu^\bullet\right)\right].$$
The latter, in turn, is equivalent to (\ref{eff-path}) for $\mathbb P$-a.e. $\omega\in\Omega$.

(ii) \ If the transport $q^\bullet$ restricted to $M\times A$ is optimal then also each of its sub-transports.
\end{proof}

Remember that due to the stationarity of $\P$ equivariance of $q^\bullet$ translates into invariance of its distribution. The next Theorem is a key step in establishing uniqueness because it shows that every optimal semicoupling is induced by a map.

\begin{theorem}\label{opt loc opt}
 Every optimal semicoupling between $\lambda^\bullet$ and $\mu^\bullet$ is locally optimal.
\end{theorem}

\begin{proof}
 Assume we are given a semicoupling $q^\bullet$ of $\lambda^\bullet$ and $\mu^\bullet$ that is equivariant but not locally optimal. According to the previous lemma, the latter implies that there is $g\in G$ and $r\in\N$ such that $q^\bullet$ is not efficient on $gB_r$, i.e.
$$\eta = \mathfrak {eff}_{gB_r}(q^\bullet)<1.$$
By invariance, this implies that $\eta = \mathfrak {eff}_{hB_r}(q^\bullet)<1$ for all $h\in G$. Hence, for any $h\in G$ there is a coupling $\tilde q^\bullet_{hB_r}$ of $\lambda^\bullet_{hB_r}$ and $1_{hB_r}\mu^\bullet$, the unique optimal coupling, which is more efficient than $1_{M\times hB_r} q^\bullet$, i.e. such that
$$ \EE[\CCost(\tilde q^\bullet_{hB_r})] \leq \eta \cdot \EE[\CCost(1_{M\times hB_r} q^\bullet)].$$
Moreover, because of the equivariance of $q^\bullet$ we have $\tilde q^\omega_{hB_r}(dx,dy)= \tilde q^{\theta_g\omega}_{ghB_r}(d(gx),d(gy))$ (see also Corollary \ref{invariance of Q}). Hence, all convex combinations of the measures $\tilde q^\omega_{hB_r}$ will have similar equivariance properties.

We would like to have the estimate above also for the restriction of $\tilde q^\bullet_{hB_r}$ to $M\times hB_0$ in order to produce a semicoupling with less transportation cost than $q^\bullet$. This is not directly possible as we cannot control the contribution of $hB_0$ to the cost of $\tilde q^\bullet_{hB_r}$. However, we can use a trick which we will also use for the construction in the next section that will give us the desired result. Remember that $\Lambda_r$ denotes the $2^r$ neighbourhood of the identity in the Cayley graph of $G$. Set
$$\bar{q}^\bullet_{hB_0} = \frac1{|\Lambda_r|}\sum_{g\in h\Lambda_r} 1_{M\times hB_0} \tilde q^\bullet_{gB_r}.$$
Then we have 
\begin{eqnarray*}
 \Cost(\bar{q}^\bullet_{hB_0})&=&  \frac1{|\Lambda_r|}\sum_{g\in h\Lambda_r} \EE\left[\int_{M\times hB_0} c(x,y) \tilde q^\bullet_{gB_r}(dx,dy)\right]\\
&=& \frac1{|\Lambda_r|}\cdot \EE\left[\int_{M\times hB_r} c(x,y) \tilde q^\bullet_{hB_r}(dx,dy)\right]\\
&=& \frac1{|\Lambda_r|}\cdot \EE\left[\CCost(\tilde q^\bullet_{hB_r})\right].
\end{eqnarray*}
The second equality holds because fixing $hB_0$ and summing over all $gB_r$ containing $hB_0$ is, due to the invariance of $\tilde q^\bullet_{gB_r}$, the same as fixing $hB_r$ and summing over all $gB_0$ contained in $hB_r$. Put $\bar{q}^\bullet = \sum_{h\in G} \bar{q}^\bullet_{hB_0}$. By construction, $\bar{q}^\bullet$ is an equivariant semicoupling of $\lambda^\bullet$ and $\mu^\bullet$. Furthermore, for any $g\in G$ we have
$$ \EE[\CCost(\bar{q}^\bullet_{gB_0})] \leq \eta \cdot \EE[1_{M\times gB_0} q^\bullet].$$
This means, that $q^\bullet$ is not asymptotically optimal.

\end{proof}

\begin{remark}
 We really need to consider $r>1$ in the above proof as it can happen that $q^\bullet$ is efficient on every fundamental region but not locally optimal. Indeed, consider $\mu=\sum_{z\in\Z^2} \delta_z$ and let $q^\bullet$ denote the coupling transporting one quarter of the Lebesgue measure of the square of edge length 2 centered at z to z. This is efficient on every fundamental region, which contains exactly one $z\in\Z^2$, but not efficient on say $[0,5)^2$.
\end{remark}

\begin{theorem}
 Assume that $\mu^\bullet$ has intensity one, then there is a unique optimal coupling between $\lambda^\bullet$ and $\mu^\bullet$.
\end{theorem}
\begin{proof}
 Assume we are given two optimal couplings $q^\bullet_1$ and $q^\bullet_2$. Then also $q^\bullet:=\frac12 q^\bullet_1+\frac12 q^\bullet_2$ is an optimal coupling because asymptotic optimality and equivariance are stable under convex combination. Hence, by the previous theorem all three couplings --  $q^\bullet_1$, $q^\bullet_2$ and $q^\bullet$ -- are locally optimal. Thus, for a.e. $\omega$ by the results of Proposition \ref{loc-opt} there exist maps $T_1^\omega,T_2^\omega,T^\omega$  such that
\begin{eqnarray*}
\delta_{T^\omega(x)}(dy)\ \lambda^\omega(dx)&=&
q^\omega(dx,dy)\\
&=&
\left(\frac12\delta_{T_1^\omega(x)}(dy)+\frac12\delta_{T_2^\omega(x)}(dy)\right)\,\lambda^\omega(dx)
\end{eqnarray*}
This, however, implies $T_1^\omega(x)=T_2^\omega(x)$ for a.e. $x\in M$. Thus $q_1^\omega=q_2^\omega$. (By Lemma \ref{semicoupling gleich coupling} we know that every invariant semicoupling between $\lambda^\bullet$ and $\mu^\bullet$ has to be a coupling.)
\end{proof}

Before we can prove the uniqueness of optimal semicouplings we have to translate Proposition \ref{fig obs} to this setting.

\begin{proposition}\label{fig obs revisited}
 Assume $\mu^\bullet$ has intensity $\beta\leq 1$ and let $q^\bullet$ be an optimal semicoupling between $\lambda^\bullet$ and $\mu^\bullet$. Let $(\pi_1)_*q^\bullet=\rho\cdot\lambda^\bullet$ for some density $\rho:\Omega\times M\to [0,1].$ Then, 
$$\tilde q^\bullet\ =\ q^\bullet + (id\times id)_*((1-\rho)\cdot\lambda^\bullet)$$
is the unique optimal coupling between $\lambda^\bullet$ and $\hat\mu^\bullet:=\mu^\bullet + (1-\rho)\cdot \lambda^\bullet.$
\end{proposition}
\begin{proof}
 Because $q^\bullet$ is equivariant by assumption also $\rho\lambda^\bullet(\cdot)=q^\bullet(\cdot,M)$ is equivariant. But then $\hat\mu^\bullet=\mu^\bullet+(1-\rho)\cdot\lambda^\bullet$ is equivariant. Moreover, by assumption we have $\mathfrak C(\tilde q^\bullet)=\mathfrak C(q^\bullet)<\infty$ which implies
$$\inf_{\kappa^\bullet\in\Pi_e(\lambda^\bullet,\hat\mu^\bullet)}\mathfrak C(\kappa^\bullet)<\infty.$$
By the previous theorem, there is a unique optimal coupling $\kappa^\bullet$ between $\lambda^\bullet$ and $\hat\mu^\bullet$ given by $\kappa^\bullet=(id,S)_*\lambda^\bullet.$ Moreover,
$$\mathfrak C(\kappa^\bullet)\ \leq\ \mathfrak C(\tilde q^\bullet)\ =\ \mathfrak C(q^\bullet).$$
Because $S_*\lambda^\bullet=\hat\mu^\bullet$ there is a density $f$ such that $S_*(f\cdot\lambda^\bullet)=(1-\rho)\cdot\lambda^\bullet.$ Indeed, for any $g\in G$ we can disintegrate 
$$1_{M\times gB_0}\kappa^\omega(dx,dy)\ =\ \kappa^{\omega,g}_y(dx)(\mu^\omega(dy)+(1-\rho^\omega(y))\lambda^\omega(dy)).$$ 
The measure  $\sum_{g\in G}\kappa^{\omega,g}_y(dx)((1-\rho^\omega(y))\lambda^\omega(dy))$ does the job.  In particular this implies that
$$\tilde \kappa^\bullet \ =\ (id\times S)_*((1-f)\cdot\lambda^\bullet)$$
is a semicoupling between $\lambda^\bullet$ and $\mu^\bullet$. The mean transportation cost of $\tilde \kappa^\bullet$ are bounded above by the mean transportation cost of $\kappa^\bullet$ as we just transport less mass. Hence, we have
$$\mathfrak C(\tilde \kappa^\bullet)\leq \mathfrak C(\kappa^\bullet)\leq \mathfrak C(\tilde q^\bullet)=\mathfrak C(q^\bullet).$$
As $q^\bullet$ was assumed to be optimal, hence asymptotically optimal, we must have equality everywhere. By uniqueness of optimal couplings this implies that $\tilde q^\bullet=\kappa^\bullet$ almost surely.
\end{proof}

\begin{lemma}
 Assume $\mu^\bullet$ has intensity $\beta\leq 1$ and  let $q^\bullet=(id,T)_*(\rho\cdot\lambda^\bullet)$ be an optimal semicoupling between $\lambda^\bullet$ and $\mu^\bullet$. Then, on the set $\{0<\rho^\omega<1\}$ we have $T^\omega(x)=x.$
\end{lemma}
\begin{proof}
 Just as in the previous proposition consider $\tilde q^\bullet=(id,S)_*\lambda^\bullet$ the optimal coupling between $\lambda^\bullet$ and $\hat\mu^\bullet.$ $\tilde q^\bullet$ is concentrated on the graph of $S$ and therefore also $q^\bullet$ has to be concentrated on the graph of $S$. In particular, this shows that $S=T$ almost everywhere almost surely (we can safely extend $T$ by $S$ on $\{\rho=0\}$). But on $\{\rho<1\}$ we have $S(x)=x$. Hence, we also have $T(x)=x$ on $\{0<\rho<1\}$.
\end{proof}

This finally enables us to prove uniqueness of optimal semicouplings.

\begin{theorem}\label{uniqueness}
 There exists at most one optimal semicoupling of $\lambda^\bullet$ and $\mu^\bullet$.
\end{theorem}

\begin{proof} Assume we are given two optimal semicouplings $q^\bullet_1$ and $q^\bullet_2$. Then also $q^\bullet:=\frac12 q^\bullet_1+\frac12 q^\bullet_2$ is an optimal semicoupling. Hence, by   Theorem \ref{opt loc opt} all three couplings --  $q^\bullet_1$, $q^\bullet_2$ and $q^\bullet$ -- are locally optimal. Thus, for a.e. $\omega$ by the results of Proposition \ref{loc-opt} there exist maps $T_1^\omega,T_2^\omega,T^\omega$  and densities $\rho_1^\omega, \rho_2^\omega, \rho^\omega$ such that
\begin{eqnarray*}
\delta_{T^\omega(x)}(dy)\,\rho^\omega(x)\, \lambda^\omega(dx)&=&
q^\omega(dx,dy)\\
&=&
\left(\frac12\delta_{T_1^\omega(x)}(dy)\rho^\omega_1(x)+\frac12\delta_{T_2^\omega(x)}(dy)\rho^\omega_2(x)\right)\,\lambda^\omega(dx)
\end{eqnarray*}
This, however, implies $T_1^\omega(x)=T_2^\omega(x)$ for a.e. $x\in \{\rho_1^\omega>0\}\cap \{\rho_2^\omega>0\}$. In particular, all optimal semicouplings are concentrated on the \emph{same} graph. To show uniqueness, we have to show that $\rho_1^\omega=\rho^\omega_2$ almost everywhere almost surely. To this end, put $A^\omega_1=\{\rho^\omega_1>\rho^\omega_2\}.$ Assume $\lambda^\omega(A_1^\omega)>0.$ On $A^\omega_1$ we have $\rho^\omega<1$. Hence, by the previous Lemma we have $T^\omega(x)=T^\omega_1(x)=T^\omega_2(x)=x.$ Similarly, on $A_2^\omega=\{\rho^\omega_2>\rho^\omega_1\}$ we have $T^\omega(x)=x.$ Hence, we have
$$(T^\omega)^{-1}(A^\omega_1)\cap A^\omega_2 = \emptyset,$$
because $A_i\subset (T^\omega)^{-1}(A_i)$ for $i=1,2$. As $q_1^\omega$ and $q_2^\omega$ are semicouplings, we must have $\mu^\omega(A)=\rho^\omega_i\cdot\lambda^\omega((T^\omega)^{-1}(A))$ for $i=1,2$ and any Borel set $A$. Putting this together gives
\begin{eqnarray*}
 \mu^\omega(A^\omega_1)& =& \rho^\omega_1\cdot\lambda^\omega((T^\omega)^{-1}(A^\omega_1))\\
&=& \rho^\omega_1\cdot\lambda^\omega((T^\omega)^{-1}(A^\omega_1)\cap A^\omega_1) + \rho^\omega_1\cdot\lambda^\omega((T^\omega)^{-1}(A^\omega_1)\cap \{\rho^\omega_1=\rho^\omega_2\})\\
&>& \rho^\omega_2\cdot\lambda^\omega((T^\omega)^{-1}(A^\omega_1)\cap A^\omega_1) + \rho^\omega_2\cdot\lambda^\omega((T^\omega)^{-1}(A^\omega_1)\cap \{\rho^\omega_1=\rho^\omega_2\})\\
&=& \mu^\omega(A^\omega_1).
\end{eqnarray*}
This is a contradiction and therefore proving $\lambda^\omega(A_1^\omega)=0$. Thus, $\rho^\omega_1=\rho^\omega_2$ almost everywhere almost surely and $q_1^\bullet=q_2^\bullet.$
\end{proof}

\subsection{Geometry of tessellations induced by fair allocations}\label{geometry of tessellations}

The fact that any optimal semicoupling is locally optimal allows us to say something about the geometries of the cells of fair allocations to point processes. The following result was already shown for probability measures in section 4 of \cite{sturm-2009} and also in \cite{142747}. We will use the representation of optimal transportation maps recalled in Remark \ref{shape of optimal transport map}.

\begin{corollary} In the case $\vartheta(r)=r^2$, given an optimal coupling $q^\bullet$ of Lebesgue measure $\leb$ and a  point process $\mu^\bullet$ of unit intensity  in $M=\R^d$ (for a Poisson point process this implies $d\geq 3$ as otherwise the mean transportation cost will be infinite, see Theorem 1.3 in \cite{huesmann2010optimal}) then for a.e. $\omega\in\Omega$ there exists a convex function $\varphi^\omega:\R^d\to\R$ (unique up to additive constants) such that
$$q^\omega=\left(Id,\nabla\varphi^\omega\right)_*\leb.$$
In particular, a 'fair allocation rule' is given by the \emph{monotone map} $T^\omega=\nabla\varphi^\omega$.

Moreover, for a.e. $\omega$ and any center $\xi\in\Xi(\omega):=\mathrm{supp}(\mu^\omega)$, the associated cell
 $$S^\omega(\xi)\ = \ {(T^\omega)^{-1}(\{\xi\})}$$
is a convex polytope of volume $\mu^\omega(\xi)\in\N$. If the point process is simple then all these cells have volume 1.
\end{corollary}

\begin{proof}
See Corollary 3.10 of \cite{huesmann2010optimal}.
\end{proof}

\begin{corollary}
 In the case $\vartheta(r)=r$, given an optimal coupling $q^\bullet$ of $m$ and a point process $\mu^\bullet$ of unit intensity on M with $\mbox{dim}(M)\geq 2$, there exists an allocation rule $T$ such that the optimal coupling is given by
$$q^\omega=\left(Id,T^\omega\right)_*m.$$

Moreover, for a.e. $\omega$ and any center $\xi\in\Xi(\omega):=\mathrm{supp}(\mu^\omega)$, the associated cell
 $$S^\omega(\xi)\ = \ {(T^\omega)^{-1}(\{\xi\})}$$
is starlike with respect to $\xi$.
\end{corollary}
\begin{proof}
By Proposition \ref{loc-opt} we know that $T^\omega=\lim_{n\to\infty}T_n^\omega$, where $T_n^\omega$ is an optimal transportation map from some set $A_n^\omega$ to $K_n$. From the classical theory (see \cite{Brenier,GangboMcCann1996}) we know that,
$$ T_n^\omega(x)=\xi_0\quad \Leftrightarrow \quad -d(x,\xi_0)+b_{\xi_0} > -d(x,\xi)+b_\xi \quad \forall \xi\in\Xi(\omega)\cap K_n, \:\xi\neq\xi_0.$$
Hence, the cell can again be written as the intersection of ``halfspaces'' $H^0_j:=\{x: - d(x,\xi_0)+b_{\xi_0} > -d(x,\xi_j)+b_{\xi_j}\}.$ Therefore, it is sufficient to show that for any $z\in H^0_j$ the whole geodesic from z to $\xi_0$ lies inside $H^0_j$. For convenience we write $\Phi_0(x)= -d(x,\xi_0)+b_{\xi_0}$ and $\Phi_j(x)=-d(x,\xi_j)+b_{\xi_j}.$ 

Assume $\xi_0\in\partial H^0_j$ and w.l.o.g. $b_{\xi_0}=0$. Then, we have
\[
 \Phi_0(\xi_0)=0=\Phi_j(\xi_0) \Rightarrow b_{\xi_j}=d(\xi_j,\xi_0).
\]
The set $N=\{z\in M: d(\xi_j,z)=d(\xi_j,\xi_0)+d(\xi_0,z)\}$ is a $m$-null set. For all $z\notin N$ we have
\[
 \Phi_j(z)=-d(\xi_j,z)+b_{\xi_j}>-d(\xi_j,\xi_0)+b_{\xi_j}-d(\xi_0,z)=\Phi_0(z)
\]
This implies that $m(T_n^{-1}(\xi_i))=0$ contradicting the assumption of $T$ being an allocation. Thus, $\xi_0\notin\partial H^0_j$ and in particular $T(\xi_0)\in \Xi=\supp(\mu).$

Assume $T(\xi_0)\neq \xi_0$. Then, there is a $\xi_j\neq \xi_0$ such that $T(\xi_0)=\xi_j$, i.e. $\Phi_j(\xi_0)=-d(\xi_0,\xi_j)+b_{\xi_j}>b_{\xi_0}=\Phi_0(\xi_0)$. Then, we have for any $p\in M, p\neq \xi_0$
\[
 -d(p,\xi_j)+b_{\xi_j}\geq -d(p,\xi_0)-d(\xi_0,\xi_j)+b_{\xi_j}>-d(p,\xi_0)+b_{\xi_0}.
\]
This implies, that $m(T^{-1}(\xi_0))= 0$ contradicting the assumption of $T$ being an allocation. Thus, $T(\xi_0)=T_n(\xi_0)=\xi_0$. 

Take any $w\in T_n^{-1}(\xi_0)$ (hence, $\Phi_0(w)>\Phi_j(w)$ for all $j\neq 0$) and $p\in M$ such that $d(\xi_0,w)=d(\xi_0,p)+d(p,w)$, i.e. $p$ lies on the minimizing geodesic from $\xi_0$ to $w$. Then, we have for any $j\neq 0$ by using the triangle inequality once more
\begin{eqnarray*}
 -d(p,\xi_0)+b_{\xi_0}&=&-d(\xi_0,w)+d(p,w)+b_{\xi_0}\\
&\geq& -d(\xi_0,w)+b_{\xi_0}+d(w,\xi_j)-d(p,\xi_j)\\
&>&-d(p,\xi_j)+b_{\xi_j},
\end{eqnarray*}
which means that $\Phi_0(p)>\Phi_j(p)$ for all $j\neq 0$. Hence, $p\in  H^0_j$ proving the claim.
\end{proof}

\begin{remark}
 i) \ Questions on the geometry of the cells of fair allocations are highly connected to the very difficult problem of the regularity of optimal transportation maps (see \cite{ma2005, loeper2009regularity, kim-2007}). The link is of course the cyclical monotonicity. The geometry of the cells of the ``optimal fair allocation'' is dictated by the cyclical monotonicity and the optimal choice of cyclical monotone map to get an asymptotic optimal coupling. 

Consider the classical transport problem between two probability measures one being absolutely continuous to the volume measure on M with full support on a convex set and the other one being a convex combination of N Dirac masses. Assume that the cell being transported to one of the N points is not connected. Then, it is not difficult to imagine that it is possible to smear out the Dirac masses slightly to get two absolutely continuous probability measures (even with very nice densities) but a discontinuous transportation map.

\medskip

ii)\ Considering $L^p$ cost on $\R^d$ with $p\notin\{1,2\}$, the cell structure is much more irregular than in the two cases considered above. The cells do not even have to be connected. Indeed, just as in the proof of the two Corollaries above it holds also for general p that $T^\omega(x)= \xi_0$ iff $\Phi_0(x)>\Phi_i(x)$ for all $i\neq 0$ where $\Phi_i(x)= - |x-\xi_i|^p+b_i$ for some constants $b_i$ (see also Example 1.6 in \cite{GangboMcCann1996}). By considering the sets $\Phi_i\equiv \Phi_0$ it is not difficult to cook up examples of probability measures such that the cells do not have to be connected. 

In the case that $p\in(0,1)$ similar to the case that $p=1$ we always have that the center of the cells lies in the cell, that is $T(\xi_i)=\xi_i$ for all $\xi_i\in\supp(\mu^\bullet)$ because the cost function defines a metric (see \cite{GangboMcCann1996}).

\medskip

iii)\ As was shown by Loeper in section 8.1 of \cite{loeper2009regularity} the cells induced by the optimal transportation problem in the hyperbolic space between an absolutely continuous measure and a discrete measure with respect to the cost function $c(x,y)=d^2(x,y)$ do not have to be connected. In the same article he shows that for the same problem on the sphere the cells have to be connected. In \cite{von2009regularity} von Nessi studies more general cost functions on the sphere, including the $L^p$ cost function $c(x,y)=d^p(x,y).$ He shows that in general for $p\neq 2$ the cells do not have to connected. This suggests that on a general manifold the cell structure will probably be rather irregular.

\end{remark}

\section{Construction}\label{s:construction}

We fix again a pair of equivariant random measures $\lambda^\bullet$ and $\mu^\bullet$ of unit resp. subunit intensity with finite optimal mean transportation cost $\mathfrak c_{e,\infty}$ such that $\lambda^\bullet$ is absolutely continuous. Additionally, we assume that $G$ satisfies some  growth condition. Recall that the $2^r$ neighbourhood of the identity element in the Cayley graph $\Delta(G,S)$ of G is denoted by $\Lambda_r$ and the range of its action on the fundamental region by $B_r$, i.e. $B_r=\bigcup_{g\in\Lambda_r} gB_0$. Then, we assume that for any $g\in G$
\begin{equation}\label{amenability}
 \frac{|\Lambda_r\triangle g\Lambda_r|}{|\Lambda_r|}\to 0 \text{ as } r\to\infty,
\end{equation}
where $|A|$ denotes the cardinality of A.  This of course implies for any $g\in G$
$$ \frac{m(B_r\triangle gB_r)}{m(B_r)}\to 0 \text{ as } r\to\infty.$$

The aim of this section is to construct the optimal semicoupling and thereby proving Theorem \ref{main thm 2}. The construction is based on approximation by semicouplings on bounded sets. We will also show a nice convergence result of these approximations, proving Theorem \ref{main thm 3}. The proofs in the first part of this section can mostly be copied from the respective results in \cite{huesmann2010optimal}. Therefore, we omit some of them and only stress those where something new happens.

\subsection{Symmetrization and Annealed Limits}
The crucial step in our construction of optimal semicouplings between $\lambda^\bullet$ and $\mu^\bullet$ is the introduction of a \emph{symmetrization} or \emph{second randomization}. We want to construct the optimal semicoupling by approximation of optimal semicouplings on bounded sets. The difficulty in this approximation lies in the estimation of the contribution of the fundamental regions $gB_0$ to the transportation cost, i.e. what does it cost to transport mass into $gB_0$? How can the cost be bounded in order to be able to conclude that the limiting measure still transports the right amount of mass into $gB_0$? The solution is to mix several optimal semicouplings and thereby get a symmetry which will be very useful (see proof of Lemma \ref{Qg-tight} (i)). One can also think of the mixing as an expectation of the random choice of increasing sequences of sets $hB_r$ exhausting $M$.

For each $g\in G$ and $r\in \N$, recall that $\Q_{gB_r}$ denotes the minimizer of $\Cost$ among the semicouplings of $\lambda^\bullet$ and $1_{gB_r}\mu^\bullet$ as constructed in Theorem \ref{eu:Q+q}. It inherits the equivariance from $\lambda^\bullet$ and $\mu^\bullet$, namely $Q_{gA}(g\cdot,g\cdot,\theta_g\omega)= Q_A(\cdot,\cdot,\omega)$  (see Corollary \ref{invariance of Q}). In particular, the stationarity of $\P$  implies $(\tau_h)_*Q_{gBr}\stackrel{d}{=}Q_{hgB_r}$. Put
$$ Q_g^r(dx,dy,d\omega)\ :=\ 1_{gB_0}(y) \frac1{|\Lambda_r|}\sum_{h\in g\Lambda_r} Q_{hB_r}(dx,dy,d\omega).$$
\begin{figure*}
\begin{center}
\begin{tikzpicture}
 \draw[thick,red] (0,0)node[anchor=north west]{$gB_0$} rectangle +(0.5,0.5);
\draw (-2,-2)node[anchor=north west]{$h_1B_r$} rectangle +(4,4);
\draw (-1,-3)node[anchor=north west]{$h_2B_r$} rectangle +(4,4);
\draw (-0.5,-1)node[anchor=north west]{$h_3B_r$} rectangle +(4,4);
\end{tikzpicture}
\end{center}
\caption{Schematic picture of the mixing procedure.}
\end{figure*}
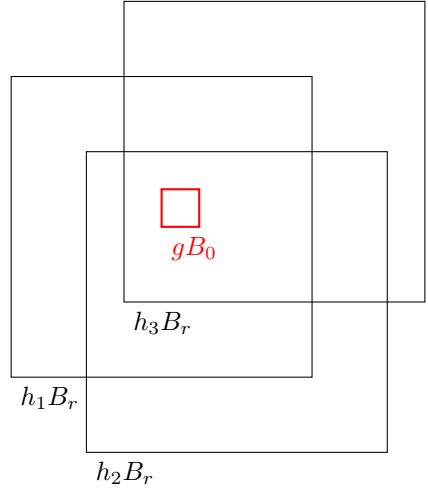

The measure $Q_g^r$ defines a semicoupling between $\lambda^\bullet$ and $1_{gB_0}\mu^\bullet$. It is a deterministic, fractional allocation in the following sense:
\begin{itemize}
\item for any $\omega$ it is a deterministic function of $\lambda^\omega$ and $\mu^\omega$ and does not depend on any additional randomness, 
\item for any $\omega$ the first marginal is absolutely continuous with respect to $\lambda^\omega$ with density $\leq 1$.
\end{itemize}
The last fact  implies that the semicoupling $\Q_g^r$ is \emph{not} optimal in general, e.g. if one transports the Lebesgue measure to a point process. The first fact implies that all the objects derived from $\Q_g^r$ in the sequel -- like $\Q_g^\infty$ and $\Q^\infty$ -- are also deterministic. Moreover, $Q_g^r$ shares the equivariance properties of the measures $Q_{hB_r}$.

\begin{lemma}\label{Qg-tight}
\begin{enumerate}
\item For each $r\in\mathbb N$ and $g\in G$
$$ \int\limits_{M\times gB_0\times \Omega}c(x,y)\Q_g^r(dx,dy,d\omega)\le {\mathfrak c_\infty}.$$
\item The family $(\Q_g^r)_{r\in\mathbb N}$ of probability measures on $M\times M\times \Omega$ is relatively compact in the weak topology.
\item There exist probability measures $\Q_g^\infty$ and a subsequence $(r_l)_{l\in\mathbb N}$ such that for all $g\in G$:
$$\Q^{r_l}_g\quad\longrightarrow\quad\Q_g^\infty\qquad\text{ weakly as $l\to\infty$.}$$
\end{enumerate}
\end{lemma} 

\begin{proof}
(i)\ Let us fix $g\in G$ and start with the important observation: \emph{For given $r\in \N$ and $g\in G$ averaging over all $h\Lambda_r$ with $h\in g\Lambda_r$ has the effect that ``$g$ attains each possible position inside $\Lambda_r$ with equal probability''} (see also the proof of Theorem \ref{opt loc opt}).

Hence, together with the invariance of $Q_{kB_r}$ we obtain
\begin{eqnarray*}
 && \int_{M\times gB_0\times\Omega} c(x,y) Q_g^r(dx,dy,d\omega)\\
&=& \frac1{|\Lambda_r|} \sum_{h\in g\Lambda_r} \int_{M\times gB_0\times\Omega} c(x,y) Q_{hB_r}(dx,dy,d\omega)\\
&=& \frac1{|\Lambda_r|} \int_{M\times gB_r\times\Omega} c(x,y) Q_{gB_r}(dx,dy,d\omega)\\
&=& \frac1{|\Lambda_r|} \CCo_{gB_r} =: \mathfrak c_r \leq \mathfrak c_\infty,
\end{eqnarray*}
by definition of $\mathfrak c_\infty$.

\medskip

(ii)\ In order to prove tightness of $(Q^r_g)_{r\in\N}$, let $(gB_0)_l$ denote the closed l--neighborhood of $gB_0$ in M. Then,
\begin{eqnarray*}
 Q^r_g(\complement(gB_0)_l, gB_0, \Omega) &\leq& \frac1{\vartheta(l)} \int_{M\times gB_0\times\Omega} c(x,y) Q_g^r(dx,dy,d\omega)\\
&\leq& \frac1{\vartheta(l)} \mathfrak c_\infty.
\end{eqnarray*}
  Since $\vartheta(l)\to\infty$ as $l\to\infty$ this proves tightness of the family $(\Q_g^r)_{r\in\mathbb N}$
on $M\times M\times\Omega$. (Recall that $\Omega$ was assumed to be compact from the very beginning.)

\medskip

(iii)\ Tightness yields the existence of $\Q_g^\infty$ and of a converging subsequence for each $g\in G$. A standard  argument ('diagonal sequence') then gives convergence for all $g\in G$ along a common subsequence (G is countable as it is finitely generated).
\end{proof}

Note that the measures $Q_g^\infty$ inherit as weak limits the property $Q_{hg}^\infty(h\cdot,h\cdot,\theta_h\cdot)=Q_g^\infty(\cdot,\cdot,\cdot)$ from the measures $Q_g^r$ (see also the proof of the equivariance property in Proposition \ref{abstract exist}).
The next Lemma allows to control the difference in the first marginals of $\Q^\infty_g$ and $\Q^\infty_{h}$ for $g\neq h$. This is the first point where we use the growth condition.

\begin{lemma}\label{Qg-disjoint}
 \begin{itemize}
  \item[i)] For all $l>0$ there exists numbers $\epsilon_r(l)$ with $\epsilon_r(l)\to 0$ as $r\to\infty$ s.t. for all $g,g^{'}\in G$ and all $r\in \N$
$$ \frac1{|\Lambda_r|} \sum_{h\in g^{'}\Lambda_r} Q_{hB_r}(A)\leq \frac1{|\Lambda_r|} \sum_{h\in g\Lambda_r} Q_{hB_r}(A) + \epsilon_r(d_\Delta(g,g^{'}))\cdot \sup_{h\in g^{'}\Lambda_r} Q_{hB_r}(A)$$
for any Borel set $A\subset M\times M\times \Omega.$
\item[ii)] For all $g_1,\ldots,g_n\in G$, all $r\in\N$ and all Borel sets $A\subset M, D\subset \Omega$
$$ \sum_{i=1}^n Q_{g_i}^r(A,M,D) \leq \left(1+\sum_{i=1}^n \epsilon_r(d_\Delta(g_1,g_i))\right) \cdot \lambda(D,A),$$
where $\lambda(D,A):= \int_D\int_A\lambda^\omega(dx)\P(d\omega).$
 \end{itemize}

\end{lemma}

\begin{proof}
 (i)\ First note that for all $g,g^{'}\in G$ and $r\in\N$ we have
$$ g^{'}\in g\Lambda_r \quad \Leftrightarrow \quad g \in g^{'}\Lambda_r.$$
In this case, for $h\in g\Lambda_r$ with $g^{'}\in h\Lambda_r$ we also have $h\in g^{'}\Lambda_r$ and  $g\in h\Lambda_r$. Moreover, 
$$ \frac{|\{h\in g\Lambda_r: g^{'} \notin h\Lambda_r\}|}{|\Lambda_r|}\leq \epsilon_r(d_\Delta(g,g^{'})), $$
for some $\epsilon_r(l)$ with $\epsilon_r(l)\to 0$ as $r\to\infty$. One possible choice for $\epsilon_r$ is
$$\epsilon_r(d_\Delta(id,g))=\frac{|\Lambda_r\triangle g\Lambda_r|}{|\Lambda_r|},$$
which tends to zero as $r$ tends to infinity for any $g\neq id$ by assumption. This implies that for each pair $g,g^{'}\in G$ and each $r\in\N$
$$\frac{|\{h\in g\Lambda_r : g^{'}\in h\Lambda_r\}|}{|\Lambda_r|} \geq 1- \epsilon_r(d_\Delta(g,g^{'})).$$
Therefore, for each Borel set $A\subset M\times M\times \Omega$
$$\frac1{|\Lambda_r|}\sum_{h\in g^{'}\Lambda_r} Q_{hB_r}(A) \leq \frac1{|\Lambda_r|} \sum_{h\in g\Lambda_r} Q_{hB_r}(A) + \epsilon_r(d_\Delta(g,g^{'}))\cdot \sup_{h\in g^{'}\Lambda_r} Q_{hB_r}(A).$$

\medskip

(ii)\ According to the previous part (i), for each Borel sets $A\subset M, D\subset \Omega$
\begin{eqnarray*}
&& \sum_{i=1}^n Q^r_{g_i}(A,M,D) \\
&=& \sum_{i=1}^n \frac1{|\Lambda_r|} \sum_{h\in g_i\Lambda_r} Q_{hB_r}(A,g_iB_0,D)\\
&\leq& \sum_{i=1}^n \left(\frac1{|\Lambda_r|} \sum_{h\in g_1\Lambda_r} Q_{hB_r}(A,g_iB_0,D) + \epsilon_r(d_\Delta(g_1,g_i))\cdot \sup_{h\in g_i\Lambda_r} Q_{hB_r}(A,g_iB_0,D)\right)\\
&\leq& \left(1+\sum_{i=1}^n \epsilon_r(d_\Delta(g_1,g_i))\right)\lambda(D,A)
\end{eqnarray*}

\end{proof}

Having these results at our hands we can copy basically line to line the respective proofs from \cite{huesmann2010optimal} (Theorem 4.3 and Corollary 4.4) to get

\begin{theorem}\label{Q infty}
 The measure $\Q^\infty:=\sum_{g\in G} Q_g^\infty$ is an optimal semicoupling of $\lambda^\bullet$ and $\mu^\bullet$.
\end{theorem}

\begin{corollary}
 (i) For $r\to\infty$, the sequence of measures $\Q^r\ :=\ \sum\limits_{g\in G} \Q^r_g$, $r\in\N$, converges vaguely to the unique optimal semicoupling $\Q^\infty$.

(ii) For each $g\in G$ and $r\in\N$ put
$$\tilde Q^r_g(dx,dy,d\omega) \ :=\  \frac1{|\Lambda_r|}\sum_{h\in g\Lambda_r} Q_{hB_r}(dx,dy,d\omega). $$
The sequence $(\tilde Q^r_g)_{r\in\N}$ converges vaguely to the unique optimal semicoupling $\Q^\infty$.
\end{corollary}

\begin{corollary}\label{lim inf change}
 Denote the set of all semicouplings of $\lambda^\bullet$ and $\mu^\bullet$ by $\Pi_s$. Then it holds
\begin{eqnarray*} \lefteqn{\inf_{q^\bullet\in \Pi_s} \liminf_{r\to\infty}\frac1{m(B_r)} \EE\left[\int_{M\times B_r} c(x,y) q^\bullet(dx,dy)\right]}\\
& = &\liminf_{r\to\infty}\inf_{q^\bullet\in \Pi_s}\frac1{m(B_r)} \EE\left[\int_{M\times B_r} c(x,y) q^\bullet(dx,dy)\right].
 \end{eqnarray*}

In particular, we have
$$\mathfrak c_\infty \ =\ \inf_{q^\bullet\in\Pi_s}\mathfrak C(q^\bullet)\ =\ \inf_{q^\bullet\in\Pi_{es}}\mathfrak C(q^\bullet)\ =\ \mathfrak c_{e,\infty}.$$
\end{corollary}
\begin{proof}
For any semicoupling $q^\bullet$ we have due to the supremum in the definition of $\mathfrak C(\cdot)$ that
$$\liminf_{r\to\infty}\frac1{m(B_r)} \EE\left[\int_{M\times B_r} c(x,y) q^\bullet(dx,dy)\right]\ \leq\ \mathfrak C(q^\bullet).$$
Hence, the left hand side is bounded from above by $\inf_{q^\bullet\in \Pi_s} \mathfrak C(q^\bullet).$ However, we just constructed an equivariant semicoupling, the unique optimal semicoupling $Q^\infty$ which attains equality, i.e. with $Q^\infty=q^\bullet\P$
$$\liminf_{r\to\infty}\frac1{m(B_r)} \EE\left[\int_{M\times B_r} c(x,y) q^\bullet(dx,dy)\right]\ =\ \mathfrak C(Q^\infty).$$
Hence, the left hand side equals $\inf_{q^\bullet\in \Pi_s} \mathfrak C(q^\bullet).$\\
The right hand side equals $\liminf_{r\to\infty}\mathfrak c_r$ which is bounded by $\mathfrak c_\infty=\inf_{q^\bullet\in \Pi_s} \mathfrak C(q^\bullet)$ by Lemma \ref{super}. By our construction, the asymptotic transportation cost of $Q^\infty$ are bounded by the right hand side, i.e.
$$\mathfrak C(Q^\infty)\leq \liminf_{r\to\infty}\mathfrak c_r$$
by Lemma \ref{Qg-tight}. Hence, also the right hand side equals $\inf_{q^\bullet\in \Pi_s} \mathfrak C(q^\bullet).$ Thus, we have equality.
\end{proof}

\begin{remark}\label{rem inf eq inf}
i) Because of the uniqueness of the optimal semicoupling the limit of the sequence $Q^r$ does not depend on the choice of fundamental region. The approximating sequence $(Q^r)_{r\in\N}$ does of course depend on $B_0$ and also the choice of generating set S that defines the Cayley graph.

\medskip

ii) In the construction of the semicoupling $Q^\infty$ we only used finite transportation cost, invariance of $Q_A$ in sense that $(\tau_h)_*Q_A\stackrel{d}{=} Q_{hA}$ and the amenability assumption on G. The only specific property of $\lambda^\bullet$ and $\mu^\bullet$ that we used is the uniqueness of the semicoupling on bounded sets which makes is easy to choose a good optimal semicoupling $Q_{gB_r}$. Hence, we can use the same algorithm to construct an optimal coupling between two arbitrary random measures. In particular this shows, that $\mathfrak c_\infty=\mathfrak c_{e,\infty}$ (see also Proposition \ref{abstract exist}).

Indeed, given two arbitrary equivariant measures $\nu^\bullet$ and $\mu^\bullet$ of unit respectively subunit intensity. For any $r\in\N$ let $Q_{B_r}=q_{B_r}^\bullet\P$ be an optimal semicoupling between $\nu^\bullet$ and $1_A\mu^\bullet$. In particular, we made some measurable choice of optimal semicoupling for each $\omega$ (they do not have to be unique), e.g. like in Corollary 5.22 of \cite{villani2009optimal}. \emph{Define} $Q_{gB_r}$ via $q_{gB_r}^{\theta_g\omega}(d(gx),d(gy)):= q^\omega_{B_r}(dx,dy).$ Due to equivariance, this is again a measurable choice of optimal semicouplings. Stationarity of $\P$ implies  $(\tau_h)_*Q_{B_r}\stackrel{d}{=} Q_{hB_r}$. Hence, by the same construction there is some optimal semicoupling $Q^\infty$ of $\nu^\bullet$ and $\mu^\bullet$ with cost bounded by $\mathfrak c_\infty.$

\end{remark}

\subsection{Quenched Limits}

According to section \ref{s:u}, the unique optimal semicoupling between $\lambda^\bullet$ and $\mu^\bullet$ can be represented on $M\times M\times\Omega$ as
$$Q^\infty(dx,dy,d\omega)=\delta_{T(x,\omega)}(dy)\,\rho^\omega(x)\lambda^\omega(dx)\,\P(d\omega)$$
by means of a measurable map
$$T:M\times\Omega\to M\cup\{\eth\},$$
defined uniquely almost everywhere and a density $\rho^\omega$. Similarly, for each $g\in G$ and $r\in\N$ there exists a measurable map
$$T_{g,r}:M\times\Omega\to M\cup\{\eth\}$$
and a density $\rho_{g,r}^\omega$ such that the measure
$$Q_{gB_r}(dx,dy,d\omega)=\delta_{T_{g,r}(x,\omega)}(dy)\,\rho_{g,r}^\omega\lambda^\omega(dx)\,\P(d\omega)$$
on
$M\times M\times\Omega$ is the unique optimal semicoupling
of
$\lambda^\bullet$ and $1_{gB_r}\,\mu^\bullet$.

\begin{proposition}\label{t-map convergence}
 For every $g\in G$
$$T_{g,r}(x,\omega)\quad\to \quad T(x,\omega)\quad\mbox{as}\quad r\to\infty\quad \mbox{in } \lambda^\bullet\otimes\P \mbox{-measure}.$$
\end{proposition}

The claim  relies on  the following two Lemmas. For the first one we use the growth assumption once more. The second one is a slight modification (and  extension) of a result in \cite{Ambrosio-ln-ot}.

\begin{lemma}\label{ugly proof}
 \begin{itemize}
  \item[i)] Fix $\omega\in\Omega$ and take two disjoint bounded Borel sets $A,B\subset M$. Let $q_A^\omega=(id,T_A^\omega)_*(\rho_A^\omega\lambda^\omega)$ be the optimal semicoupling between $\lambda^\omega$ and $1_A\mu^\omega.$  Similarly, let $q_B^\omega$ and $q_{A\cup B}^\omega$ be the unique optimal semicouplings between $\lambda^\omega$ and $1_B\mu^\omega$ respectively $1_{A\cup B}\mu^\omega$ with transport maps $T^\omega_B$ and $T^\omega_{A\cup B}$ and densities $\rho^\omega_B$ and $\rho^\omega_{A\cup B}$. Then, it holds that
$$\rho_{A\cup B}^\omega(x)\geq \max\{\rho^\omega_A(x),\rho^\omega_B(x)\} \quad \lambda^\omega a.s..$$
\item[ii)] For any $g\in G$ and $r\in\N$ we have $\rho^\omega_{g,r}(x)\leq \rho^\omega(x) \quad (\lambda^\bullet\otimes\P)$ a.s..
\item[iii)] For any $g\in G$  we have $\lim_{r\to\infty}\rho^\omega_{g,r}(x)\nearrow \rho^\omega(x) \quad (\lambda^\bullet\otimes\P)$ a.s..
 \end{itemize}
\end{lemma}

\begin{proof}
 i)\ Firstly, note that if $\{\rho^\omega_A>0\}\cap\{\rho^\omega_B>0\}=\emptyset$ we have $\rho^\omega_{A\cup B}=\rho^\omega_A+\rho_B^\omega.$ Because of the symmetry in A and B it is sufficient to prove that $\rho_{A\cup B}^\omega\geq \rho^\omega_B.$ The proof is rather technical and involves an iterative choice of possibly different densities.

For simplicity of notation we will suppress $\omega$ and write $f=\rho_B$ and $h=\rho_{A\cup B}$ and $T=T_B, S=T_{A\cup B}$. We will show the claim by contradiction. Assume there is a set D of positive $\lambda$ measure such that $f(x)> h(x)$ on D. Put $f_+:=(f-h)_+$ and $\mu_1:= T_*(f_+\lambda).$ Let $h_1\leq h$ be such that $S_*(h_1\lambda)=\mu_1$, that is $h_1$ is a subdensity of $h$ such that $T_*(f_+\lambda)=S_*(h_1\lambda)$ (for finding this density we can use disintegration as in the proof of Proposition \ref{fig obs revisited}). 

If $1_{\{h_1>0\}}h>f$ on some set $D_1$ of positive $\lambda$ measure, we are done. Indeed, as $f$ is the  unique $\CCost$ minimizing choice for the semicoupling between $\lambda$ and $1_B\mu$ the transport $S_*(1_{D_1}h_1\lambda)=:\tilde \mu_1$ must be more expensive than the respective transport $T_*(1_{\tilde D_1} f_+\lambda)=\tilde \mu_1$ for some suitable set $\tilde D_1$. Hence, $q_{A\cup B}$ cannot be minimizing and therefore not optimal, a contradiction.

If $1_{\{h_1>0\}}h\leq f$ we can assume wlog that $T_*(h_1\lambda)=\mu_2$ and $\mu_1$ are singular to each other. Indeed, if they are not singular we can choose a different $h_1$ because $1_B\mu$ has to get its mass from somewhere. To be more precise, if $\tilde h\leq h_1$ is such that $T_*(\tilde h\lambda)\leq\mu_1$ we have $T_*((f_+ + \tilde h)\lambda)>\mu_1.$ Therefore, there must be some density $h'$ such that $h'+h_1\leq h$ and $S_*((h'+h_1)\lambda)=T_*((f_+ + \tilde h)\lambda).$ Because, $f_+>0$ on some set of positive measure and $T_*(f\lambda)\leq S_*(h\lambda)$, there must be such an $h_1$ as claimed.

Take a density $h_2\leq h$ such that $S_*(h_2\lambda)=\mu_2.$ If $1_{\{h_2>0\}}h>f$ on some set $D_2$ of positive $\lambda$ measure, we are done. Indeed, the optimality of $q_B$ implies that the choice of $f_+$ and $h_1$ is cheaper than the choice of $h_1$ and $h_2$ for the transport into $\mu_1+\mu_2$ (or maybe subdensities of these).

If $1_{\{h_2>0\}}h\leq f$ and $\{h_2>0\}\cap\{f_+>0\}$ has positive $\lambda$ measure, we get a contradiction of optimality of $q_{A\cup B}$ by cyclical monotonicity. Otherwise, we can again assume that $T_*(h_2\lambda)=:\mu_3$ and $\mu_2$ are singular to each other. Hence, we can take a density $h_3\leq h$ such that $S_*(h_3\lambda)=\mu_3$. 

Proceeding in this manner, because $f_+\lambda(M)=h_i\lambda(M)>0$ for all i and the finiteness of $q_B(M,M)$ one of the following two alternatives must happen
\begin{itemize}
 \item there is $j$ such that $1_{\{h_j>0\}}h > f$ on some set of positive $\lambda$ measure.
\item there are $j\neq i$ such that $\{h_j>0\}\cap\{h_i>0\}$ on some set of positive $\lambda$ measure with $f_+=h_0$.
\end{itemize}
Both cases lead to a contradiction by using the optimality of $q_B$, either by producing a cheaper semicoupling (in the first case) or by arguing via cyclical monotonicity (in the second case).

\medskip

ii)\ Fix $\omega, g$ and r. Denote the density of the first marginal of $\tilde Q^l_f$ by $\zeta^\omega_{f,l}.$ It is a convex combination of $\rho^\omega_{h,l}$ with $h\in f\Lambda_l$. For $h\in G$ with $d(g,h)\leq n$ we have $g\Lambda_r \subset h\Lambda_{r+n}$. Hence, we have $\rho^\omega_{g,r}\leq \rho^\omega_{h,r+n}$ by the first part of the Lemma. Therefore, the contribution of $\rho^\omega_{g,r}(x)$ to $\zeta^\omega_{g,r+n}(x)$ is at least the number of $h\in G$ such that $d(g,h)\leq n$ divided by $|\Lambda_{r+n}|$. Hence,
$$\frac{|\Lambda_n|}{|\Lambda_{r+n}|}\rho^\omega_{g,r}(x) \leq \zeta^\omega_{g,r+n}(x).$$
By the assumption (\ref{amenability}) we have
$$ \lim_{r\to\infty} \frac{|K\Lambda_r\triangle\Lambda_r|}{|\Lambda_r|}=0,$$
for any finite $K\subset G$. If we take $K=\{ h:d(h,id)=r\}$ we can conclude
$$ \frac{|\Lambda_{n+r}|}{|\Lambda_{n}|} \leq 1+ \frac{|K\Lambda_n\triangle\Lambda_n|}{|\Lambda_n|}\to 1 \text{ as } n\to\infty.$$
Fix $\epsilon>0$. If $\rho^\omega_{g,r} >\epsilon+ \rho^\omega$ on some positive $(\lambda^\bullet\otimes\P)-$ set, we have that $\zeta^\omega_{g,r+n}(x)>\rho^\omega(x) + \epsilon/2$ on some positive $(\lambda^\bullet\otimes\P)-$ set for all $n$ such that $\frac{|\Lambda_n|}{|\Lambda_{n+r}|}\geq 1-\epsilon/2,$ because $\rho^\omega_{g,r}\leq 1$ and thus $\rho^\omega\leq 1-\epsilon.$ Denote this set by $A$, so $A\subset M\times \Omega.$ Then, we have $\tilde Q^{r+n}_g(A\times M)> Q^\infty(A\times M)+\epsilon/2$ for all n big enough. However, this is a contradiction to the vague convergence of $\tilde Q^r_g$ to $\Q^\infty$ which was shown in the last section.

\medskip

iii)\ The last part allows to interpret $\rho^\omega_{g,r}$ as a density of $(\rho^\omega\lambda^\omega)$ instead of as a density of $\lambda^\omega.$ We will adopt this point of view and show that $\rho^\omega_{g,r}$ converges to 1 $(\lambda^\bullet\otimes \P)$ a.s..

Assume that $\rho^\omega_{g,r}(x)\leq \gamma<1$ for all $r\in\N$. Moreover, assume that there is $k\in G$ and $s\in\N$ such that $\rho^\omega_{k,s}(x)>\gamma.$ Then there is a $t\in\N$ such that $g\Lambda_t\supset k\Lambda_s$. The first part of the Lemma then implies that $\rho^\omega_{g,t}(x)\geq \rho^\omega_{k,s}(x)>\gamma$ which contradicts the assumption of $\rho^\omega_{g,r}(x)\leq \gamma$. Hence, if we have $\rho^\omega_{g,r}(x)\leq \gamma<1$ for all $r\in\N$ on a set of positive $(\lambda^\bullet\otimes\P)$ measure we must have $\rho^\omega_{k,s}(x)\leq\gamma$ for all $k\in G$ ans $s\in\N$ on this set. Denote this set again by $A$, $A\subset M\times \Omega.$ As $\zeta^\omega_{g,r}$ is a convex combination of the densities $\rho^\omega_{h,r}$ it must also be bounded away from 1 by $\gamma$ on the set A. However, this is again a contradiction to the vague convergence of $\tilde Q^r_g$ to $\Q^\infty$.

\end{proof}

\begin{lemma}\label{ae-convergence}
Let $X, Y$ be locally compact separable spaces,
  $\theta$  a Radon measure on $X$
 and $\rho$ a metric on $Y$ compatible with the topology.

 (i) For all $n\in\N$ let $T_n,T:X\to Y$ be Borel measurable maps. Put $Q_n(dx,dy):=\delta_{T_n(x)}(dy)\theta(dx)$ and $Q(dx,dy):=\delta_{T(x)}(dy)\theta(dx)$. Then,
$$ T_n\to T \:\: \mbox{ locally in measure on }X \quad \Longleftrightarrow\quad Q_n\to Q \mbox{ vaguely in }\mathcal M(X\times Y).$$

(ii) More generally, let $T$ and $Q$ be as before whereas
$$Q_n(dx,dy):=\int_{X'}\delta_{T_n(x,x')}(dy)\,\theta'(dx')\,\theta(dx)$$
for some probability space
$(X',\frak A',\theta')$ and
suitable measurable maps $T_n: X\times X'\to Y$. Then
$$Q_n\to Q \mbox{ vaguely in }\mathcal M(X\times Y)\quad \Longrightarrow\quad
 T_n(x,x')\to T(x) \:\: \mbox{ locally in measure on }X\times X'.$$
\end{lemma}

For a proof we refer to section 4 of \cite{huesmann2010optimal}

\begin{proof}[Proof of the Proposition]
 Firstly, we will show that the Proposition holds for  'sufficiently many' $g\in G$. We want to apply the previous Lemma. Recall that 
$$\tilde Q^r_g \to\Q^\infty \quad \text{ vaguely on } M\times M\times\Omega,$$
where
$$\Q^\infty(dx,dy,d\omega) = \delta_{T(x,\omega)}(dy) \rho^\omega(x)\lambda^\omega(dx)\P(d\omega)$$
and
$$\tilde Q^r_g(dx,dy,d\omega)\ =\ \frac{1}{|\Lambda_r|} \sum_{h\in gB_r} Q_{hB_r}(dx,dy,d\omega)\ =\ \frac{1}{|\Lambda_r|} \sum_{h\in gB_r} \delta_{T_{h,r}(x)}(dy) \rho_{h,r}^\omega(x)\lambda^\omega(dx)\P(d\omega),$$
with transport maps $T, T_{h,r} : M\times \Omega\to M\cup \{\eth\}$ and densities $\rho, \rho_{h,r}:M\times\Omega\to\R_+$. The Lemma above allows to interpret $\rho_{h,r}$ as density of the measure $\rho\lambda^\bullet$. Fix $k\in G$ and let $\theta_r'$ be the uniform measure on $k\Lambda_r$. Take $\theta=\rho\lambda^\bullet\otimes\P, X=M\times \Omega$ and $Y=M\cup\{\eth\}$. Apply the same reasoning as in the proof of the second assertion in the last lemma, however, now with changing $\theta'$, to get
\begin{equation}\label{theta theta cvg}
 \lim_{r\to\infty} (\theta\otimes \theta_r')\left(\left\{(x,\omega,h)\in \tilde K\times G\ :\ \rho^\omega_{h,r}(x)\cdot d(T_{h,r}(x,\omega),T(x,\omega))\geq \epsilon\right\}\right)=0.
\end{equation}
Let $H\subset G$ be those h for which 
$$ \lim_{r\to\infty}\theta\left(\left\{(x,\omega)\in \tilde K \ :\ \rho^\omega_{h,r}(x)d(T_{h,r}(x,\omega),T(x,\omega))\geq \epsilon\right\}\right)>0.$$
Because we know that (\ref{theta theta cvg}) holds, we must have $\lim_{r\to\infty} \theta'_r(H)=0$. Hence, there are countably many $g\in G$ such that 
$$  \lim_{r\to\infty}\theta\left(\left\{(x,\omega)\in \tilde K \ :\ d(T_{g,r}(x,\omega),T(x,\omega))\geq \epsilon\right\}\right)=0,$$
where we used that $\rho^\omega_{g,r}\nearrow 1$ for $(\lambda^\bullet\otimes\P)$ a.e. $(x,\omega)$, according to the Lemma above. This shows that the Proposition holds for those $g$. 

Pick one such $g\in G$. Then the first part of the previous lemma implies
$$ Q_{gB_r} \to Q^\infty \quad \text{ vaguely on } M\times M\times \Omega.$$
This in turn implies that for any $h\in G$ $(\tau_h)_* Q_{gB_r} \to (\tau_h)_* Q^\infty \stackrel{(d)}= Q^\infty$ by invariance of $Q^\infty.$ Moreover, by Corollary \ref{invariance of Q} we have $(\tau_h)_* Q_{gB_r} \stackrel{(d)}= Q_{hgB_r}$. This means, that for any $h\in G$ we have 
$$ Q_{hgB_r} \to Q^\infty\quad \text{ vaguely on } M\times M\times \Omega.$$
Applying once more the first part of the previous Lemma proves the Proposition.
\end{proof}

\begin{corollary}
 There is a measurable map $\Psi:\mathcal M(M)\times \mathcal M(M)\to\mathcal M(M\times M)$ s.t. $q^\omega:=\Psi(\lambda^\omega,\mu^\omega)$ denotes the unique optimal semicoupling between $\lambda^\omega$ and $\mu^\omega$. In particular the optimal semicoupling is a factor.
\end{corollary}
\begin{proof}
 We showed that the optimal semicoupling $Q^\infty$ can be constructed as the unique limit point of a sequence of deterministic functions of $\lambda^\bullet$ and $\mu^\bullet$. Hence, the map $\omega\mapsto q^\omega$ is measurable with respect to the sigma algebra generated by $\lambda^\bullet$ and $\mu^\bullet$. Thus, there is a measurable map $\Psi$ such that $q^\bullet=\Psi(\lambda^\bullet,\mu^\bullet).$
\end{proof}

\subsubsection{Semicouplings of $\lambda^\bullet$ and a point process.}

If $\mu^\bullet$ is known to be a point process the above convergence result can be significantly improved. Just as in Theorem 4.8 and Corollary 4.9 of \cite{huesmann2010optimal} we get

\begin{theorem}
 For any $g\in G$ and every bounded Borel set $A\subset M$
$$\lim_{r\to\infty} (\lambda^\bullet\otimes \P)\left(\left\{ (x,\omega)\in A\times \Omega\  :\ T_{g,r}(x,\omega) \neq T(x,\omega) \right\}\right)\ =\ 0.$$
\end{theorem}

\begin{corollary} There exists a subsequence $(r_l)_l$ such that
$$T_{g,r_l}(x,\omega)\quad\to \quad T(x,\omega)\qquad\mbox{as}\quad l\to\infty$$
for almost every $x\in M$, $\omega\in\Omega$ and every $g\in G$.
Indeed, the sequence $(T_{g,r_l})_l$ is finally stationary. That is,  there exists a random variable $l_g:M \times\Omega\to\N$ such that almost surely
$$T_{g,r_l}(x,\omega)\quad = \quad T(x,\omega)\qquad \mbox{for all }\ l\ge l_g(x,\omega).$$
\end{corollary}

\section{The other semicouplings}\label{s:other}
In the previous sections we studied semicouplings between two equivariant random measures $\lambda^\bullet$ and $\mu^\bullet$ with intensities 1 and $\beta\leq 1$ respectively.
In this section we want to remark on  the case that $\mu^\bullet$ has intensity $\beta>1$. Then, $q^\bullet$ is a semicoupling between $\lambda^\bullet$ and $\mu^\bullet$ iff for all $\omega\in\Omega$
$$ (\pi_1)_*q^\omega=\lambda^\omega\quad \text{ and }\quad (\pi_2)_*q^\omega\leq\mu^\omega.$$
This will complete the picture of semicouplings with one marginal being absolutely continuous. In the terminology of section \ref{section 2.2} we should better talk about semicouplings between $\mu^\bullet$ and $\lambda^\bullet$. However, we prefer to keep $\lambda^\bullet$ as first marginal as it better suits our intuition of transporting a continuous quantity somewhere. We will only prove the key technical lemma, existence and uniqueness of optimal semicouplings on bounded sets. From that result one can deduce following the reasoning of the previous sections the respective results on existence and uniqueness for optimal semicouplings. We will not give the proofs because they are completely the same or become easier as we do not have to worry about densities.

\begin{lemma}
 Let $\rho\in L^1(M,m)$ be a nonnegative density. Let $\mu$ be an arbitrary measure on M with $\mu(M)\geq (\rho\cdot m)(M).$ Then, there is a unique semicoupling $q$ between $(\rho\cdot m)$ and $\mu$ minimizing $\CCost(\cdot)$. Moreover, $q=(id,T)_*(\rho\cdot m)$ for some measurable cyclically monotone map $T$.
\end{lemma}
\begin{proof}
 The existence of one $\CCost$ minimizing semicoupling $q$ goes along the same lines as for example in Lemma \ref{leb-dirac}. Let $q_1$ be one such minimizer. As $q_1$ is minimizing it has to be an optimal coupling between its marginals. Therefore, it is induced by a map, that is $q_1=(id,T_1)_*(\rho\cdot m)$. Let $q_2=(id,T_2)_*(\rho\cdot m)$ be another minimizer. Then, $q_3=\frac12(q_1+q_2)$ is minimizing as well. Hence, $q_3=(id,T_3)_*(\rho\cdot m).$ However, just as in the proof of Lemma \ref{leb-dirac} this implies $T_1=T_2$ $(\rho m)$ almost everywhere and therefore $q_1=q_2$.
\end{proof}

\section{Cost estimate for Compound Poisson processes}\label{s: cost estimate}
In this section we state some cost estimates for the transport between the Lebesgue and a $\gamma-$compound Poisson process. We consider $M=\R^d, \lambda^\bullet=\mathcal L$ the Lebesgue measure and $\mu^\bullet$ a $\gamma-$compound Poisson process of intensity one with iid weights $(X_i)_{i\in\N}, X_1 \sim \gamma$ In \cite{huesmann2010optimal} a general technique was developed which allows to deduce upper estimates on the transportation cost by upper moment estimates of the random variable $\mu^\bullet(A)$. In short, having good bounds on moments and inverse moments of $\mu^\bullet$ allows to deduce transportation cost estimates.

We have the following estimates on $L^p-$ cost
\begin{proposition}
 \begin{itemize}
  \item[i)] Let $p>1$ be such that  $\EE[X_1^p]=\infty$ and $\vartheta(r)\geq r^{(p-1)d}$, then $\mathfrak c_\infty=\infty.$ 
  \item[ii)] Assume $d\le 2$ and $\EE[X_1^2]<\infty$. Then for any concave $\hat\vartheta:[1,\infty)\to\R$ dominating $\vartheta$
$$\int_1^\infty \frac{\hat\vartheta(r)}{r^{1+d/2}} \d r<\infty
\quad \Longrightarrow \quad {\mathfrak c}_\infty<\infty.$$
\item[iii)] Assume $d\geq 3$ and $p<p_0-1$ with $2<p_0=\sup\{q: \EE[X_1^q]<\infty\}<\infty$. Then for $\vartheta(r) = r^p$ we have $\mathfrak c_\infty <\infty.$
\item[iv)] Assume $d\geq 3$ and $\EE[X_1^p]<\infty$ for all $p>0$. Then for $\vartheta(r)=r^p$  we have $\mathfrak c_\infty<\infty$ for any $p>0.$
 \end{itemize}
\end{proposition}
\begin{proof}
ad i)  $\mathfrak c_\infty$ can easily be bounded from below by the cost of transporting mass $X_1$ optimally into a single point. This transportation cost behaves for fixed $X_1=X_1(\omega)$ like
$$\int_{0}^{cX_1^{1/d}}\vartheta(r) r^{d-1}dr\ \gtrsim \ X_1^p.$$
Taking expectation wrt $X_1$ yields the desired result. 
\medskip\\
The other claims are straightforward adaptations of the techniques from \cite{huesmann2010optimal}. We omit the details.
\end{proof}

Mark\'o and Timar \cite{marko2011poisson} constructed an allocation of Lebesgue measure to a Poisson point process in dimensions $d\geq 3$ with optimal tail behavior. In our language this means that there is a constant $\kappa$ such that the optimal mean transportation between a Lebesgue measure and a Poisson point process with cost function $c(x,y)=\exp(\kappa|x-y|^d)$ are finite in dimensions $d\geq 3.$ Up to the constant $\kappa$ this is optimal (e.g. see \cite{huesmann2010optimal}). Their construction is based on an algorithm by Ajtai, Koml\'os and Tusn\'ady \cite{Ajtai-K-T} and uses two key properties of the Poisson point process, independence on disjoint sets and exponential concentration around the mean in big cubes. In the case of a $\gamma-$compound Poisson process the independence is inherited from the Poisson process. If we take $\gamma$ to be the exponential distribution one can show 

\begin{lemma}
 Let $Z=\sum_{i=1}^NX_i$ with $N$ a Poisson random variable with mean $\alpha$ and $(X_i)_{i\in\N}$ a sequence of iid exponentially distributed random variables with mean 1 independent of $N$. For any $0<\rho<1$ it holds that
 $$\P[|Z-\alpha|>\alpha \rho] \leq 2\cdot\exp(-\alpha(2+\rho-2\sqrt{1+\rho}))\leq 2\cdot\exp\left(-\alpha\left(\frac{\rho^2}{4}-\frac{\rho^3}{8}\right)\right).$$
\end{lemma}

Hence, by using the very same algorithm as in \cite{marko2011poisson} one gets for $\gamma$ the exponential distribution with mean one
\begin{proposition}
 Let $d\geq 3$ and $\mu^\bullet$ a $\gamma-$compound Poisson process. Then there is constant $\kappa>0$ such that optimal mean transportation cost between $\mathcal L$ and $\mu^\bullet$ for the cost function $c(x,y)=\exp(\kappa\cdot |x-y|^d)$ is finite.
\end{proposition}

\section{Stability}\label{s: stability}
As an application of the previous results, especially the existence and uniqueness results, we want to study stability properties of the optimal coupling between two random measures. Moreover, we will show some metric properties of the mean transportation cost. 
\medskip\\
Given sequences of random measures $(\lambda^\bullet_n)_{n\in\N}, (\mu^\bullet_n)_{n\in\N}$ and their optimal couplings $(q^\bullet_n)_{n\in\N}$ we want to understand which kind of convergence $\lambda^\bullet_n\to\lambda^\bullet, \mu^\bullet_n\to\mu^\bullet$ implies the convergence $q^\bullet_n\to q^\bullet,$ where $ q^\bullet$ denotes the/an optimal coupling between $\lambda^\bullet$ and $\mu^\bullet$. If for all n $\lambda_n, \lambda, \mu_n, \mu$ are probability measures $q_n$ the optimal coupling between $\lambda_n$ and $\mu_n$ (all transportation cost involved bounded by some constant) and $\lambda_n\to\lambda, \mu_n\to\mu$ weakly, then, by the classical theory (see Theorem 5.20 in \cite{villani2009optimal}), also along a subsequence $q_n\to q$ weakly, where $q$ is an optimal coupling between $\lambda$ and $\mu$.
\medskip\\
A naive approach to our problem would be to ask for $\lambda^\bullet_n\stackrel{d}{\to}\lambda^\bullet$ and $\mu^\bullet_n\stackrel{d}{\to}\mu^\bullet$. However, in this case let $\lambda^\bullet$ and $\mu^\bullet$ be two independent Poisson point process and set $\lambda^\bullet_n\equiv \mu^\bullet_n \equiv \lambda^\bullet$. Then, we indeed have $\lambda^\bullet_n\stackrel{d}{\to}\lambda^\bullet$ and $\mu^\bullet_n\stackrel{d}{\to}\mu^\bullet$. Yet, the optimal couplings $(q^\bullet_n)_{n\in\N}$, which are just $q^\omega_n(dx,dy)=\delta_x(dy)\lambda_n^\omega(dx),$ do not converge to any coupling between $\lambda^\bullet$ and $\mu^\bullet$ in any reasonable sense. Moreover, the couplings $(q^\bullet_n)_{n\in\N}$ do 'converge' (they are all the same) to some coupling $\tilde q^\bullet$ with marginals being $\lambda^\bullet$ and $\lambda^\bullet$ having the same \emph{distribution} as $\lambda^\bullet$ and $\mu^\bullet$.
\medskip\\
The next best guess, instead of vague convergence in distribution is vague convergence on $M\times M\times \Omega$. Together with some integrability condition this will be the answer if the cost of the couplings converge.

\medskip

For two random measure $\lambda^\bullet, \mu^\bullet$ with intensity one and $c(x,y)=d^p(x,y)$ with $p\in [1,\infty)$ write 
$$\W_p^p(\lambda^\bullet, \mu^\bullet)=\inf_{q^\bullet\in \Pi_{es}(\lambda^\bullet,\mu^\bullet)}\mathfrak C(q^\bullet)= \inf_{q^\bullet\in \Pi_{es}(\lambda^\bullet,\mu^\bullet)} \EE\left[\int_{M\times B_0} d^p(x,y)\ q^\bullet(dx,dy)\right].$$
We want to establish a triangle inequality for $\W_p$ and therefore restrict to $L^p$ cost functions. We could also extend this to more general cost functions by using Orlicz type norms as developed in \cite{sturm2011generalized}. However, to keep notations simple we stick to this case.
\medskip\\
In this chapter, we will assume that all pairs of random measures considered will be equivariant and modeled on the same probability space $(\Omega,\mathfrak A,\P).$ As usual $\P$ is assumed to be stationary. Moreover, we will always assume without explicitly mentioning it that the mean transportation cost is finite.

\medskip

Recall the disintegration Theorem \ref{disintegration theorem}. This will allow us to use the gluing lemma.

\begin{proposition}\label{metric}
 Let $\mu^\bullet,\lambda^\bullet,\xi^\bullet$ be three  equivariant random measures of unit intensity.
\begin{itemize}
 \item[i)] $\W_p(\lambda^\bullet,\mu^\bullet)=0\quad \Leftrightarrow \quad \lambda^\omega=\mu^\omega \quad \P-a.s..$
\item[ii)] $\W_p(\lambda^\bullet,\mu^\bullet)=\W_p(\mu^\bullet,\lambda^\bullet).$
\item[iii)] $\W_p(\lambda^\bullet,\mu^\bullet)\leq \W_p(\lambda^\bullet,\xi^\bullet)+\W_p(\xi^\bullet,\mu^\bullet).$
\end{itemize}

\end{proposition}

\begin{proof}
 i)\ $\W_p(\lambda^\bullet,\mu^\bullet)=0$ iff there is a coupling of $\lambda^\bullet$ and $\mu^\bullet$ which is entirely concentrated on the diagonal almost surely, that is iff $\lambda^\omega=\mu^\omega$ $\P-$ almost surely.

\medskip

ii)\ Let $q^\bullet$ be an optimal coupling between $\lambda^\bullet$ and $\mu^\bullet$. By definition
$$\W^p_p(\lambda^\bullet,\mu^\bullet) = \EE\left[\int_{M\times B_0} d^p(x,y) q^\bullet(dx,dy)\right].$$
For $g,h\in G$ put
 $$f(g,h)=\EE\left[\int_{gB_0\times hB_0} d^p(x,y) q^\bullet(dx,dy)\right].$$
By equivariance and stationarity, we have $f(g,h)=f(kg,kh)$ for all $k\in G$. Hence, we can apply the mass transport principle.
$$\sum_{h\in G}f(g,h)=\EE\left[\int_{gB_0\times M} d^p(x,y) q^\bullet(dx,dy)\right]=\sum_{g\in G}f(g,h)= \EE\left[\int_{M\times hB_0} d^p(x,y) q^\bullet(dx,dy)\right].$$
This proves the symmetry.

\medskip

iii)\ The random measures are random variables on some Polish space. Therefore, we can use the gluing Lemma (cf. \cite{dudley2002real} or \cite{villani2009optimal}, chapter 1) to construct an equivariant random measure $q^\bullet$ on $M\times M \times M$ such that 
$$ (\pi_{1,2})_*q^\bullet \in \Pi_{opt}(\lambda^\bullet,\mu^\bullet) \quad \text{and} \quad (\pi_{2,3})_*q^\bullet \in \Pi_{opt}(\mu^\bullet,\xi^\bullet),$$
where $\Pi_{opt}(\lambda^\bullet,\mu^\bullet)$ denotes the set of all optimal couplings between $\lambda^\bullet$ and $\mu^\bullet$. $q^\bullet$ is equivariant as the optimal couplings are equivariant and $q^\bullet$ is glued together along the common marginal of these two couplings. 

To be more precise let $q^\bullet_1\in \Pi_{opt}(\lambda^\bullet,\mu^\bullet) $ and $ q^\bullet_2\in \Pi_{opt}(\mu^\bullet,\xi^\bullet)$. Then, consider $1_{M\times gB_0\times \Omega} q^\bullet_1$ and $1_{gB_0\times M\times \Omega}q^\bullet_2$ to produce with the usual gluing Lemma a measure $q^\bullet_g$ on $M\times M\times M\times\Omega$ with the desired marginals on $M\times gB_0 \times M\times \Omega$. As all these sets are disjoint we can add up  the different $q^\bullet_g$ yielding $q^\bullet=\sum_{g\in G} q^\bullet_g$ a measure with the desired properties.

For $g,h\in G$ put 
$$e(g,h)=\EE\left[\int_{M\times gB_0\times hB_0} d^p(x,z) q^\bullet(dx,dy,dz)\right].$$
By equivariance of $q^\bullet$, we have $e(kg,kh)=e(g,h)$ for all $k\in G.$ By the mass transport principle this implies
$$\EE\left[\int_{M\times B_0\times M} d^p(x,z) q^\bullet(dx,dy,dz)\right]=\EE\left[\int_{M\times M \times B_0} d^p(x,z) q^\bullet(dx,dy,dz)\right].$$
Then we can conclude, using the Minkowski inequality
\begin{eqnarray*}
\lefteqn{ \W_p(\lambda^\bullet,\xi^\bullet)}\\
&\leq& \EE\left[\int_{M\times M \times B_0} d^p(x,z) q^\bullet(dx,dy,dz)\right]^{1/p}\\
&=& \EE\left[\int_{M \times B_0 \times M} d^p(x,z) q^\bullet(dx,dy,dz)\right]^{1/p}\\
&\leq & \EE\left[\int_{M \times B_0 \times M} d^p(x,y) q^\bullet(dx,dy,dz)\right]^{1/p} + \EE\left[\int_{M \times B_0 \times M} d^p(y,z) q^\bullet(dx,dy,dz)\right]^{1/p}\\
&=& \W_p(\lambda^\bullet,\mu^\bullet) + \W_p(\mu^\bullet,\xi^\bullet).
\end{eqnarray*}
In the last step we used the symmetry shown in part ii).

\end{proof}

\begin{remark}
 Note that the first two properties also hold for general cost functions and general semicouplings. The assumption of equal intensity is not needed for these statements.
\end{remark}

Let $\mathcal P_p=\{ \text{equivariant random measures } \mu^\bullet: \W_p(m,\mu^\bullet)<\infty \}.$

\begin{proposition}\label{half stability}
 Let $(\mu^\bullet_n)_{n\in\N}, \mu^\bullet \in\mathcal P_p$ be  random measures of intensity one. Let $q^\bullet_n$ denote the optimal coupling between m and $\mu^\bullet_n$ and $q^\bullet$ the optimal coupling between $m$ and $\mu^\bullet$. Consider the following  statements.
\begin{itemize}
 \item[i)] $\W_p(\mu^\bullet_n,\mu^\bullet)\to 0$ as $n\to\infty.$
\item[ii)] $\mu^\bullet_n\P \to \mu^\bullet\P$ vaguely   and $\W_p(\mu^\bullet_n, m)\to \W_p(\mu^\bullet, m)$ as $n\to\infty.$
\item[iii)]  $q^\bullet_n\P \to q^\bullet\P$ vaguely  and $\W_p(\mu^\bullet_n, m)\to \W_p(\mu^\bullet, m)$ as $n\to\infty.$
\item[iv)] $q^\bullet_n\P \to q^\bullet\P$ vaguely  and
 $$\lim_{R\to\infty}\limsup_{n\to\infty} \EE\left[\int_{(\complement(B_0)_R)\times B_0} d^p(x,y) q^\bullet_n(dx,dy)\right]=0,$$
 where $(B_0)_R$ denotes the $R-$neighbourhood of $B_0$. 
\end{itemize}
Then i) implies ii). iii) and iv) are equivalent and either of them implies i).
\end{proposition}
\begin{proof}
$i) \Rightarrow ii):$\ For any $f\in C_c(M\times \Omega)$ we have to show that $\lim_{n\to\infty}\EE[\mu_n(f)-\mu(f)]=0.$ To this end, fix $f\in C_c(M\times \Omega)$ such that $\supp(f)\subset K\times\Omega$ for some compact set $K$.  $f$ is uniformly continuous. Let $\eta>0$ be arbitrary and set $\epsilon=\eta/(2m(K))$. Then, there is $\delta$ such that $d(x,y)\leq\delta$ implies $d(f(x,\omega),f(y,\omega))\leq\epsilon.$ Put $ A=\{(x,y):d(x,y)\geq \delta\}\cap M\times K$ and denote by $\kappa^\bullet_n$ an optimal coupling between $\mu^\bullet_n$ and $\mu^\bullet.$ By assumption, there is $N\in\N$ such that for all $n>N$ we have $\W_p^p(\mu^\bullet_n,\mu^\bullet)\leq \frac{ \eta \delta^p}{4 \|f\|_\infty m(k)}.$ Then, we can estimate for $n>N$
\begin{eqnarray*}
 \left|\EE[\mu_n^\omega(f)-\mu^\omega(f)]\right|&\leq&  \left|\EE\left[\int_{M\times M}(f(x,\omega)-f(y,\omega))\kappa^\omega_n(dx,dy)\right]\right|\\
&\leq & \epsilon\cdot m(K) + \left| \EE\left[\int_{A}(f(x,\omega)-f(y,\omega))\kappa^\omega_n(dx,dy)\right]\right|\\
&\leq& \frac\eta2 +  2 \|f\|_\infty \EE\left[\kappa^\bullet_n(A)\right]\\
&\leq& \frac\eta2 +  2 \|f\|_\infty \frac1{\delta^p} \W_p^p(\mu^\bullet_n,\mu^\bullet)\cdot m(K)\\
&\leq&\frac\eta2 +\frac\eta2 =\eta,
\end{eqnarray*}
The second assertion in $ii)$ is a direct consequence of the triangle inequality:
\begin{eqnarray*}
 \W_p(\mu^\bullet_n,m)\leq \W_p(\mu^\bullet_n,\mu^\bullet)+\W_p(\mu^\bullet,m)
\end{eqnarray*}
and
\begin{eqnarray*}
 \W_p(\mu^\bullet,m)\leq \W_p(\mu^\bullet_n,\mu^\bullet)+\W_p(\mu^\bullet_n,m).
\end{eqnarray*}
Taking limits yields the claim.

 \medskip

$iii) \Leftrightarrow iv):$\  By the existence and uniqueness result we know that $q_n^\omega(dx,dy)=\delta_{T^\omega_n(x)}(dy)m(dx)$ and $q^\omega(dx,dy)=\delta_{T^\omega(x)}(dy)m(dx).$ In particular, we have $\mu^\omega_n(dx)\P(d\omega)=(T^\omega_n)_*m(dx)\P(d\omega).$ By Lemma \ref{ae-convergence} we know that the vague convergence of $q^\bullet_n\P\to q^\bullet\P$ implies that $T_n\to T$ locally in $m\otimes\P$ measure. This in turn implies the convergence of $f\circ (id,T_n)\to f\circ (id,T)$ in $m\otimes \P$ measure for any continuous and compactly supported function $f:M\times M\to\R.$ Then, it follows as in the proof of Lemma \ref{ae-convergence} that
$$\EE\int f(x,T_n(x))m(dx) \to \EE\int f(x,T(x)) m(dx).$$
Let $c_k(x,y)$ be a continuous compactly supported function such that for any $(x,y)\in (B_0)_{k-1}\times B_0$ we have $d^p(x,y)=c_k(x,y),$ for any $x\in \complement(B_0)_k$ we have $c_k(x,y)=0$ and $c_k(x,y)\leq d^p(x,y)$ for all $(x,y)\in M\times M$. Then, we have
\begin{eqnarray*}
&& \limsup_{n\to\infty} \EE\left[\int_{\complement((B_0)_R)\times B_0} d^p(x,y) q^\bullet_n(dx,dy)\right]\\
 &\leq &  \limsup_{n\to\infty} \left(\EE\left[\int_{M\times B_0} d^p(x,y) q^\bullet_n(dx,dy)\right] - \EE\left[\int_{M\times B_0} c_R(x,y) q^\bullet_n(dx,dy)\right]\right)\\
&=& \EE\left[\int_{M\times B_0} d^p(x,y) q^\bullet(dx,dy)\right] - \EE\left[\int_{M\times B_0} c_R(x,y) q^\bullet(dx,dy)\right]\\
&\leq &\EE\left[\int_{\complement((B_0)_{R-1}\times B_0} d^p(x,y) q^\bullet(dx,dy)\right].
\end{eqnarray*}
Taking the limit of $R\to\infty$ proves the implication $iii)\Rightarrow iv)$. The other direction is similar.

\medskip

$iv) \Rightarrow i):$\ We will show that $\W_p(\mu^\bullet_n,\mu^\bullet)\to 0$ by constructing a not optimal coupling between $\mu^\bullet_n$ and $\mu^\bullet$ whose transportation cost converges to zero. Let $T_n,T$ be the transportation maps from the previous steps. Put $Q_n(dx,dy):=(T_n,T)_*m.$ This is an equivariant coupling of $\mu^\bullet_n$ and $\mu^\bullet$ because the maps $T_n, T$ are equivariant in the sense that (see also Example \ref{equivariance of maps})
$$ T^{\theta_g\omega}(x)=gT^\omega(g^{-1}x).$$
The transportation cost are given by
$$\mathfrak C(Q_n) = \EE\left[\int_{B_0\times M} d^p(x,y) Q_n(dx,dy)\right]=\EE\left[\int_{B_0} d^p(T_n(x),T(x)) m(dx)\right].$$
We want to divide the integral into four parts. Put $A^R=\{x:d(T(x),x)\geq R\}$ and similarly $A_n^R=\{x:d(T_n(x),x)\geq R\}$. The four parts will be the integrals over $B_0\cap \complement^aA_n^R\cap \complement^bA^R$ with $a,b\in \{0,1\}$ and $\complement^0A=A.$   We estimate the different integrals separately. 
\begin{eqnarray*}
 \EE\left[\int_{B_0\cap \complement A_n^R \cap \complement A^R} d^p(T_n(x),T(x))m(dx)\right]\to 0,
\end{eqnarray*}
by a similar argument as in the previous step due to the convergence of $T_n\to T$ locally in $m\otimes \P$ measure and the boundedness of the integrand. 
\begin{eqnarray*}
&& \EE\left[\int_{B_0\cap  A_n^R \cap  A^R} d^p(T_n(x),T(x))m(dx)\right]\\
&\leq& 2^p\ \EE\left[\int_{B_0\cap  A^R} d^p(x,T(x))m(dx)\right] + 2^p\ \EE\left[\int_{B_0\cap  A_n^R} d^p(x,T_n(x))m(dx)\right].
\end{eqnarray*}
If $d(x,y)\leq R, d(x,z)\geq R$ and $d(y,z)\leq d(x,z) + R + a$ for some constant a($=\mbox{diam}(B_0)$), there is a constant $C_1$, e.g. $C_1=2+\mbox{diam}(B_0),$ such that $d(y,z)\leq C_1 d(x,z)$ (because $d(x,z)+R+a\leq (2+a)d(x,z)).$ This allows to estimate with $(x=x, T(x)=z, T_n(x)=y)$
\begin{eqnarray*}
 \EE\left[\int_{B_0\cap \complement A_n^R \cap  A^R} d^p(T_n(x),T(x))m(dx)\right]\leq C_1^p\ \EE\left[\int_{B_0\cap  A^R} d^p(x,T(x))m(dx)\right].
\end{eqnarray*}
Similarly
\begin{eqnarray*}
 \EE\left[\int_{B_0\cap  A_n^R \cap \complement A^R} d^p(T_n(x),T(x))m(dx)\right]\leq C_1^p\ \EE\left[\int_{B_0\cap  A_n^R} d^p(x,T_n(x))m(dx)\right].
\end{eqnarray*}
This finally gives
\begin{eqnarray*}
&& \limsup_{n\to\infty} \EE\left[\int_{B_0\times M} d^p(x,y) Q_n(dx,dy)\right]\\
&\leq & \lim_{R\to\infty}  \limsup_{n\to\infty} \Bigg(2^p\ \EE\left[\int_{B_0\cap  A^R} d^p(x,T(x))m(dx)\right] + 2^p\ \EE\left[\int_{B_0\cap  A_n^R} d^p(x,T_n(x))m(dx)\right]\\
&& +\ C_1^p\ \EE\left[\int_{B_0\cap  A^R} d^p(x,T(x))m(dx)\right] + C_1^p\ \EE\left[\int_{B_0\cap  A_n^R} d^p(x,T_n(x))m(dx)\right]\Bigg)\\
&=& 0,
\end{eqnarray*}
by assumption.
\end{proof}

\begin{remark}
 For an equivalence of all statements we would need that ii) implies iii). In the classical theory this is precisely the stability result (Theorem 5.20 in \cite{villani2009optimal}). This result is proven by using the characterization of optimal transports by cyclical monotone supports. However, as mentioned in the discussion on local optimality (see Remark \ref{remark on loc opt}) a cyclical monotone support is not sufficient for optimality in our case. 
\end{remark}

We do not have real stability in general but we get at least close to it.

\begin{proposition}\label{nearly stability}
 Let $(\lambda^\bullet_n)_{n\in\N}$ and $(\mu^\bullet_n)_{n\in\N}$ be two sequences of equivariant random measures. Let $q_n^\bullet$ be the unique optimal coupling between $\lambda^\bullet_n$ and $\mu^\bullet_n$. Assume that $\lambda^\bullet_n\P\to\lambda^\bullet\P$ vaguely, $\mu^\bullet_n\P\to\mu^\bullet\P$ vaguely and $\sup_n \mathfrak C(q_n^\bullet)\leq c<\infty$. Then, there is an equivariant coupling $q^\bullet$ of $\lambda^\bullet$ and $\mu^\bullet$ and a subsequence $(q^\bullet_{n_k})_{k\in\N}$ such that $q_{n_k}^\bullet\P\to q^\bullet\P$ vaguely, the support of $q^\bullet$ is cyclically monotone and 
$$\mathfrak C(q^\bullet)\leq\liminf_{n\to\infty}\mathfrak C(q_n^\bullet).$$
In particular, if 
  $$\lim_{n\to\infty}\mathfrak C(q^\bullet_n)=\inf_{\tilde {q}^\bullet \in \Pi_{es}(\lambda^\bullet,\mu^\bullet)} \mathfrak C(\tilde {q}^\bullet)$$
 $q^\bullet$ is the/an optimal coupling between $\lambda^\bullet$ and $\mu^\bullet$ and $q^\bullet_n\P\to q^\bullet\P$ vaguely.
\end{proposition}

The proof is basically the same as for Proposition \ref{abstract exist}. Hence, we omit the details.

\begin{remark}
 The last proposition also holds if we  consider semicouplings instead of couplings (see Proposition \ref{abstract exist}).
\end{remark}

\begin{example}[Wiener mosaic]
 Let $\mu^\bullet_0$ be a Poisson point process of intensity one on $\R^3$. Let each atom of $\mu_0$ evolve according to independent Brownian motions for some time t. The resulting discrete random measure is again a Poisson point process, denoted by $\mu^\bullet_t$ (e.g. see page 404 of \cite{doob1953stochastic}). Consider the transport problem between the Lebesgue measure $\mathcal L$ and $\mu^\bullet_t$ with cost function $c(x,y)=|x-y|^2$. Let $q_t^\bullet$ be the unique optimal coupling between $\mathcal L$ and $\mu^\bullet_t$. Then, $\mathfrak C(q_t^\bullet)=\W_2(\mathcal L,\mu^\bullet_t)=\W_2(\mathcal L,\mu^\bullet_s)$ for any $s\in\R$ as  $\mu^\bullet_s$ and $\mu^\bullet_t$ are both Poisson point processes of intensity one. Moreover, we clearly have $\mu^\bullet_s\P\to\mu^\bullet_t\P$ vaguely as $s\to t$ and therefore $q_s^\bullet\P\to q_t^\bullet\P$ vaguely. By Lemma \ref{ae-convergence}, this implies the convergence of the transport maps $T_s\to T_t$ locally in $\mathcal L\otimes \P$ measure. In particular, we get a continuously moving mosaic.
\end{example}

\begin{example}[Voronoi tessellation]
 Let $\mu^\bullet$ be a simple point process of unit intensity. Put $\mu^\bullet_\beta=\beta\cdot \mu^\bullet$. We want to consider semicouplings between the Lebesgue measure $\leb$ and $\mu^\bullet_\beta$ for $\beta>1$.  By the results of section \ref{s:other}, there is a unique optimal semicoupling $q^\bullet_\beta$ between $\mathcal L$ and $\mu^\bullet_\beta.$  Moreover, $q_\beta^\omega=(id,T_\beta^\omega)_*\mathcal L$ for some measurable map $T_\beta:\R^d\times \Omega\to \R^d.$ It is clear that $q^\omega_\infty$ induces the Voronoi tessellation with respect to the support of $\mu^\omega_1$ no matter which cost function $\vartheta$ we consider. We want to show that $q^\bullet_\beta\to q^\bullet_\gamma$ vaguely as $\beta\to\gamma$ for large $\gamma$. For this it is sufficient to show that $C:\beta \mapsto \mathfrak C(q^\bullet_\beta)$ is continuous. For $\gamma>\beta>1$ note that $q^\bullet_\beta$ is also a semicoupling of $\mathcal L$ and $\mu^\bullet_\gamma$. Hence, $C(\beta)$ is monotonously decreasing in $\beta$. Moreover, from the previous Lemma we know that 
$$\mathfrak C(q^\bullet_\gamma)\leq \liminf_{\beta\to\gamma} \mathfrak C(q^\bullet_\beta).$$
Therefore, $C(\beta)$ is a monotonously decreasing lower semicontinuous function. This implies that it is right continuous. With a  bit of work it is also possible to show that $C(\cdot)$ is left continuous in $\beta < \infty.$ For $\beta=\infty$ one can  show directly that $q^\bullet_\beta\to q^\bullet_\infty$ vaguely. 

 \end{example}

\end{document}